%% file: Mpq.tex
\documentclass[11 pt,a4paper,dvips]{article}
\usepackage{amsmath,amssymb,amsfonts,amsthm,epsfig,graphics,graphicx}

\usepackage[usenames,dvipsnames]{color}
 \def\NN{{\mathbb N}}  
\def\QQ{{\mathbb Q}} \def\RR{{\mathbb R}}  
   
 \def\ZZ{{\mathbb Z}}

\def\cC{{\cal C}}    
    
    \def\cW{{\cal W}}
\def\cF{{\cal F}}  \def\cL{{\cal L}}

\newtheorem{theorem}{{Theorem}}[section]
\newtheorem{proposition}[theorem]{{Proposition}}

\newtheorem{lemma}[theorem]{{Lemma}}

\newtheorem{corollary}[theorem]{{Corollary}}

\theoremstyle{definition}
\newtheorem{definition}[theorem]{{Definition}}

\theoremstyle{remark}
\newtheorem{remark}[theorem]{{Remark}}

\title{Anosov flows on Dehn surgeries on the figure-eight knot}
\author{Bin Yu}
\date{\today}

\begin{document}

\maketitle

\begin{abstract}
 The purpose of this paper is to  classify Anosov flows on the $3$-manifolds obtained by Dehn surgeries on the figure-eight knot. This set of  $3$-manifolds is denoted by $\{M(r)\mid r\in \QQ\}$, which contains
 the first class of hyperbolic $3$-manifolds  admitting Anosov flows
  in history, discovered by Goodman.
 Combining with the classification of Anosov flows on the sol-manifold $M(0)$ due to Plante, we have:
\begin{enumerate}
  \item  if $r\in \ZZ$, up to topological equivalence,  $M(r)$  exactly carries a unique Anosov flow,
  which is constructed by Goodman by doing a Dehn-Fried-Goodman surgery on a suspension Anosov flow;
  \item  if $r\notin \ZZ$, $M(r)$  does not carry any Anosov flow.
\end{enumerate}
As a consequence of the second result, we  get infinitely many closed orientable hyperbolic $3$-manifolds which carry taut foliations but does not carry any Anosov flow.
The fundamental tool in the proofs is the set of branched surfaces built by Schwider, which is
used to carry essential laminations on $M(r)$.
\end{abstract}

\input{intro}

\input{Schwider}

\input{Schwiderlist}
\input{trori}

\input{Anosov}

\input{proof}

\input{biblio}
\vskip 1cm

\noindent Bin Yu

\noindent {\small School of Mathematical Sciences}

\noindent{\small Tongji University, Shanghai 200092, CHINA}

\noindent{\footnotesize{E-mail: binyu1980@gmail.com }}

\end{document}

%% file: intro.tex
\section{Introduction}\label{s.int}
Anosov flows generalize the geodesic flows on closed Riemannian manifold with negative curvature
by an important property of these geodesic flows: the whole manifold is a hyperbolic set of the flow. In his celebrated paper \cite{An}, Anosov proved that every Anosov flow is both structure stable and ergodic. To be precise, let $X$ be a nonsingular $C^r$ ($r\geq 1$) vector field on a
closed Riemannian manifold $W$, $X$ is called an Anosov vector field if there exists an $X$-invariant splitting $TM =E^s \oplus \RR X\oplus E^u$ and some constant $C>0$, $\lambda >0$
such that:
 \begin{eqnarray*}
 \|D X_t (v)\| & \leq & C e^{-\lambda t} \|v\|, \mbox{ for any } v\in E^s \mbox{, } t\geq 0;\\
 \|D X_{-t} (v)\| & \leq & C e^{-\lambda t} \|v\|, \mbox{ for any } v\in E^u \mbox{, } t\geq 0.
 \end{eqnarray*}
  The corresponding
flow $X_t$ is called an Anosov flow.

It is natural to qualitative understanding of these flows, many works have been done in this direction, for instance, \cite{Fen1},  \cite{BF}, \cite{Bar}, \cite{Gh}, \cite{Pl}, etc. Nevertheless, even in dimension $3$, a complete classification seems to be absolutely out of reach. But for certain classes of $3$-manifolds, there are already complete classifications, for instance:
\begin{itemize}
  \item Plante \cite{Pl} and Ghys \cite{Gh}  classified Anosov flows on torus bundle over circle and circle bundle over surfaces respectively.
  \item Barbot \cite{Bar} classified Anosov flows on a class of graph manifolds, which he called generalized Bonatti-Langevin manifolds.
\end{itemize}
More recently, in \cite{YY}, Yang and the author of this paper classified non-transitive Anosov flows
on the toroidal manifolds by gluing two figure-eight knot complements. But it still is open that whether these manifolds carry  transitive Anosov flows.

Notice that all of the underling $3$-manifolds considered above are toroidal.
As far as the author knows, there does not exist any complete classification result about Anosov flows
on hyperbolic $3$-manifolds  before.
The first class of Anosov flows on hyperbolic $3$-manifolds was constructed by Goodman \cite{Goo}
by doing a type of dynamical Dehn surgery, namely Dehn-Fried-Goodman surgery, along a periodic orbit of a suspension Anosov flow $X_t^0$. The suspension Anosov flow $X_t^0$  is induced by the vector field $(0, \frac{\partial}{\partial t})$ on the sol-manifold $M(0) = T^2 \times [0,1]/ (x,1)\sim (A(x),0)$,  where $A=\left(
                          \begin{array}{cc}
                            2 & 1 \\
                            1 & 1 \\
                          \end{array}
                        \right)$ is an Anosov automorphism on $T^2$. Note that  $M(0)$ is the sol manifold endowed with a torus fibration over $S^1$  with
 monodromy map $A=\left(
                          \begin{array}{cc}
                            2 & 1 \\
                            1 & 1 \\
                          \end{array}
                        \right).$

Dehn-Fried-Goodman surgery  is a  powerful and classical technique to build new Anosov flow from old one, which was introduced by Goodman \cite{Goo} and Fried \cite{Fri}.\footnote{In fact, Goodman's surgery  and Fried's surgery were introduced   in two different ways.
  An Anosov flow obtained by doing Goodman's surgery is smooth,  but is not easy to compare the behaviours of the flowlines with the old Anosov flow. Conversely, an Anosov flow obtained by doing Fried's surgery only is  a topological Anosov flow, but is easy to compare  the behaviours of the flowlines with the old Anosov flow. In his thesis \cite{Sha},  Shannon   proved that up to topological equivalence, these two surgeries are the same.}
  One can find the related  details in \cite{Goo}, \cite{Fri} and \cite{Sha}. Here we only roughly describe it. Let $\phi_t$ be a transitive Anosov flow on a closed $3$-manifold $W$ and $\alpha$ be a periodic orbit of $\phi_t$. Dehn-Fried-Goodman surgery can ensure that for every $k\in \ZZ$, one can build an Anosov flow $\phi_t^k$ on $W(\alpha,k)$ so that $\phi_t^k$ and $\phi_t$ are topologically equivalent by withdrawing $\alpha$ in both of the manifolds before and after the surgery. Here $W(\alpha, k)$ is the $3$-manifold obtained by doing  $k$-Dehn surgery on $W$ along $\alpha$.

Now let us explain Goodman's examples more precisely.
Let  $\gamma$ be the periodic orbit of $(M(0), X_t^0)$
associated to the origin $O$ which is a fixed point of $A=\left(
                          \begin{array}{cc}
                            2 & 1 \\
                            1 & 1 \\
                          \end{array}
                        \right)$, and $V$ be a small tubular neighborhood of $\gamma$.
Further we set $N= M(0) - \mbox{int}(V)$ which is homeomorphic to the figure-eight knot complement with a torus boundary $T$,
where $\mbox{int}(V)$ means the interior of $V$.
Up to isotopy, there is  a unique circle $l$ and a unique circle $m$
on $T$ so that  $l$ bounds a once-punctured
torus in $N$,  and we can get $S^3$ if we glue $N$ to $V$ along their boundaries by sending $m$
to a meridian circle of $V$. We fix two orientations on $m$ and $l$.
\begin{definition}\label{d.Mr}
For every $r=\frac{q}{p}\in \QQ$, $M(r)$ is defined to be the manifold obtained by filling  the solid torus $V$ on the figure-eight knot complement $N$ so that  a circle in $T=\partial N$ parameterized by $pl +q m$  bounds a disk in $V$.
\end{definition}
 Note that $\{M(r)\mid r\in \QQ\}$
is the set of the $3$-manifolds obtained by doing Dehn surgeries on the figure-eight knot.
For every $r\in \ZZ$,  Goodman constructed her Anosov flow $X_t^r$ on $M(r)$ by doing $r$-Dehn-Fried-Goodman surgery
along the periodic orbit $\gamma$ of $X_t^0$. A classical result due to Thurston \cite{Thu} is that  $M(r)$ is  hyperbolic except when $r\in \{0, \pm1, \pm2, \pm3, \pm4\}$. Therefore, each $X_t^r$ ($|r|>4$) is an Anosov flow on a hyperbolic $3$-manifold $M(r)$ so that they are two oriented circles.

 Each Anosov flow $X_t^r$ ($r\in \ZZ$ and $|r|>4$) shares
several impressive properties, for instance,
\begin{enumerate}
  \item by Fenley \cite{Fen1}, $X_t^r$ is  skew $\RR$-covered;
  \item by Fenley \cite{Fen1} and Barthelme-Fenley \cite{BaF},  each periodic orbit of
$X_t^r$ is  isotopic to infinitely many periodic orbits of $X_t^r$.
\end{enumerate}

\subsection{Main results}\label{ss.mainresult}
It is natural to ask if there exists any other Anosov flows on some $M(r)$ with $r\in \ZZ$, and
if there exists any Anosov flows on some $M(r)$ with $r\in \QQ\setminus\ZZ$.
These questions motivated  to classify Anosov flows on the manifolds set $\{M(r)\mid r\in \QQ\}$.
Our main result is:

\begin{theorem}\label{t.main}
Let $M(r)$ ($r\in \QQ$) be the closed $3$-manifold obtained by doing $r$-Dehn surgery on the figure-eight knot in $S^3$. Then,
\begin{enumerate}
  \item if $r\in \ZZ$, up to topological equivalence, $M(r)$ carries a unique  Anosov flow $X_t^r$;
  \item if $r\notin \ZZ$, $M(r)$ does not carry any Anosov flow.
\end{enumerate}
\end{theorem}

\begin{remark}\label{r.main-item1}
 \begin{enumerate}
   \item  Item $1$ of Theorem \ref{t.main} provides  the complete classification about Anosov flows on an infinitely many hyperbolic $3$-manifolds set $\{M(r)\mid r\in \ZZ \setminus \{0, \pm1, \pm2, \pm3, \pm4\}\}$. To the best of our knowledge, it is the first  complete classification result about Anosov flows on hyperbolic $3$-manifolds.
   \item Converse to a type of  flexibility  about Anosov flows on a class of hyperbolic $3$-manifolds indicated by a recent work of Bowden  and  Mann \cite{BM}, which shows that
for every $n\in \NN$, there exists a hyperbolic $3$-manifold which at least carries $n$ pairwise non-topologically equivalent Anosov flows, item $1$ of Theorem \ref{t.main} shows a type of rigidity about Anosov flows on the hyperbolic $3$-manifolds set $\{M(r)\mid r\in \ZZ \setminus \{0, \pm1, \pm2, \pm3, \pm4\}\}$.
 \end{enumerate}
\end{remark}

\begin{remark}\label{r.main-plante}
In Theorem \ref{t.main}, the classification of the Anosov flows on the sol-manifold $M(0)$ was done by Plante \cite{Pl} (see also Theorem \ref{t.plante} and Remark \ref{r.plante}), and we will only prove the theorem on the manifolds set $\{M(r)\mid r\in \QQ \setminus \{0\}\}$. In fact, our proof will strongly depend on the classification  on $M(0)$ due to Plante.
\end{remark}

It is also fundamental to understand which hyperbolic $3$-manifolds do not admit any Anosov flows.
One of the most natural obstructions  for  the existence of Anosov flows is  the existence of taut foliation. Let us be more precise.
The \emph{stable/unstable foliation} of a three dimensional Anosov flow (abbreviated  as \emph{Anosov foliation}) is the union of the weak stable/unstable manifolds of the Anosov flow.
Since there does not exist any compact leaf,   an Anosov foliation
always is a taut foliation. In fact, there are a decent number of other obstructions for the existence of Anosov flows coming from foliation theory, contact and symplectic geometry, for instance:
\begin{enumerate}
  \item due to its definition, Anosov foliation should does not admit any invariant transverse measure and satisfies that the Euler number of the corresponding tangent plane field is $0$;
  \item Anosov flows always can induce a type of ``supporting contact structures" (see \cite{ET}, \cite{Mi} or \cite{Ho}), which should be tight,  symplectically fillable, and have no Giroux torsion, etc. \footnote{The Reeb flows of those supporting contact structures can be chosen to be transverse to the Anosov foliations. See Example 2 of  Zung \cite{Zu} for more details about this fact. Furthermore, Zung \cite{Zu} developed an interesting theory of such foliations.}
\end{enumerate}
All of the relationships above between Anosov flow and contact and symplectic geometry can be found in Hozoori \cite{Ho}.

Nevertheless, it is still difficult to build concrete hyperbolic $3$-manifolds so that they do not admit any Anosov flows by using these obstructions. As far as the author knows, the only known examples
are the hyperbolic $3$-manifolds which do not carry any taut foliations. In history,
Roberts, Shareshian and Stein \cite{RSS} and Calegari and Dunfield \cite{CD} independently
constructed infinitely many hyperbolic $3$-manifolds as the first classes of this types of examples.
Hence, these hyperbolic $3$-manifolds do not carry  any Anosov flows.
It is natural to ask: do there exist some  hyperbolic $3$-manifolds which carry taut foliations but
do not carry any Anosov flows? In fact, people tend to believe that the answer is positive partially
because as we talked above, Anosov foliations are very special taut foliations.

 Notice that by using a kind of standard surgery in foliation theory, filling monkey saddles, it is not difficult to show that
each $M(r)$ ($r\in \QQ$) carries a taut foliation (see, for instance, Gabai \cite{Ga1}). Then,
 Item $2$ of Theorem \ref{t.main} can help us to  positively  answer the question:

\begin{corollary} \label{c.main}
There are infinitely many closed orientable hyperbolic $3$-manifolds $\{M(r)\mid r\in \QQ \setminus\ZZ\}$ so that they
carry taut foliations but do not carry Anosov flows.
\end{corollary}

\subsection{The brief introduction for the proof of Theorem \ref{t.main}} \label{ss.bintro}
The proof of Theorem \ref{t.main} strongly depends on the classification of essential laminations on Dehn surgeries on the figure-eight knot  due to Schwider \cite{Sch}. In his thesis \cite{Sch}, Schwider uses the spine decomposition on the figure-eight knot complement $N$ introduced by Thurston \cite{Thu}, and
 normal branched surface theory adapted by Brittenham \cite{Bri2} and Gabai \cite{Ga2} to find $39$ types of essential branched surfaces to carry any essential lamination on $M(r)$.

Now let us briefly introduce the strategy  about the proof of Theorem \ref{t.main}.
Let $\cF$ be  an Anosov foliation on $M(r)$. By splitting finitely many leaves of $\cF$,
we get an essential lamination $\cL$ on $M(r)$, which is called by an \emph{Anosov lamination}.
$\cL$ should be fully carried by some branched surface $B$ in Schwider's list.
Note that essentially depending on Theorem 3.1 of \cite{Bri1} (see also Theorem \ref{t.norsoltor}), each  branched surface $B$ in Schwider's list satisfies that
either each sector of $B\cap V$ is a disk, or there exists some annuli sector in $B\cap V$.
In the first case, we will show that each of this kind of branched surfaces can not fully carry any Anosov lamination by
discussing the manifold $W(B)=M-\mbox{int}(N(B))$ and using some
necessary conditions for the complement of a branched surface $B$ carrying Anosov laminations.
To each branched surface $B$ in the second case,  by some careful analysis depending on  some facts about three dimensional Anosov flows,  we will conclude that
\begin{enumerate}
  \item if $r\in \QQ\setminus \ZZ$, each $M(r)$  does not carry any Anosov lamination which is fully carried by $B$;
  \item if $r\in \ZZ\setminus \{0\}$ and $M(r)$   carries an Anosov lamination,  then there exists a periodic
  orbit $\omega$ in the associated Anosov flow $Y_t$ so that $M(r)-U(\omega)$ is homeomorphic to
  the figure-eight knot complement $N$, where $U(\omega)$ is a small open tubular neighborhood of $\omega$ in $M(r)$.
\end{enumerate}
We can see that now
item $1$  of Theorem \ref{t.main} is proved, and what left is to  show that $Y_t$ and $X_t^r$ are topologically equivalent. Finally, we will finish the proof basically depending on the
classification of Anosov flows on sol-manifolds due to Plante \cite{Pl}, and
the classification of expanding attractors on the figure-eight knot complement $N$ due to Yang and the author \cite{YY}.

\subsection{Outline of the article}\label{ss.outline}
The paper is organized as follows.
 For convenience to a reader, in Section \ref{s.Schwider},
we will introduce some backgrounds about Schwider's classification
of essential laminations on $M(r)$.
In Section \ref{s.Schwiderlist},
we will introduce  Schwider's $39$ branched surfaces and the corresponding background $3$-manifolds. In this section, we will also  survey the idea and the points behind  Schwider's construction.
In Section \ref{s.trori}, we will introduce  transverse orientations of
branched surfaces, and prove that some branched surfaces in Schwider's list are transversely orientable.
In Section \ref{s.Anosov}, we will introduce some facts about three dimensional Anosov flows,  and we will also  get some results as the preparations to the proof of Theorem \ref{t.main}.
In the last section (Section  \ref{s.proof}), under the preparations above, we will finish the proof of Theorem \ref{t.main}.

\section*{Acknowledgments}
We would like thank Yi Shi,   Pierre Dehornoy and Surena Hozoori for their many helpful comments and suggestions. In particular, we thank Yi Shi for his very valuable suggestions about the proof of Lemma \ref{l.infper}, and Surena Hozoori for his numerous and helpful suggestions to add some appropriate remarks about the relationships between Anosov flow and contact topology. 
The author is supported by the National Natural Science Foundation of China (NSFC 11871374).

%% file: Schwider.tex
\section{Normal form of essential laminations on $M(r)$}\label{s.Schwider}
In this section, we will introduce some backgrounds about Schwider's classification of essential laminations on $M(r)$. The basic knowledge about foliation and lamination theory can be found in
\cite{Cal}.

\subsection{Essential laminations and essential branched surfaces}
Essential laminations were introduced as a generalization of incompressible surfaces and taut foliations
by Gabai and Oertel \cite{GO}:
\begin{definition}\label{d.eslam}
Let $\cL$ be a two dimensional lamination in a closed $3$-manifold $M$.
We say that $\cL$ is \emph{essential} if it satisfies:
\begin{enumerate}
  \item the inclusion of each leaf of the lamination into $M$ induces an injection on $\pi_1$;
  \item the path closure of $M-\cL$, $M_{\cL}$ is irreducible;
  \item $\cL$ has no sphere leaves;
  \item $\cL$ is end-incompressible.
\end{enumerate}
\end{definition}

A \emph{disk-with-end} $D$ is a disk with a closed arc removed from its boundary.
Let $\partial D$ be the open arc of the boundary which remains.
If for every proper embedding $d: (D, \partial D)\to (M_{\cL}, \partial M_{\cL})$,
there is a proper embedding $d': (D, \partial D)\to  \partial M_{\cL}$
with $d=d'$ on $\partial D$. Then we say that the lamination $\cL$ is \emph{end-incompressible}.
Here the properness of the embedding of $D$ means that the end of $D$ is mapped to an end of
$M_{\cL}$.

One of the main results of \cite{GO} is that an essential lamination can be carried by
a special type of compact two dimensional complexes, say essential branched surfaces.  Let us be more precise.

A \emph{branched  surface} $B$ in a closed $3$-manifold $M$ is  a compact two complex which can be decomposed into the union of finitely many
compact smooth embedded surfaces, locally modeled by Figure \ref{f.locbranch}.
The union of the branched points is called by the \emph{branched locus} of $B$,
which is the union of finitely many smooth immersed circles in $M$.
\begin{figure}[htp]
\begin{center}
  \includegraphics[totalheight=2.8cm]{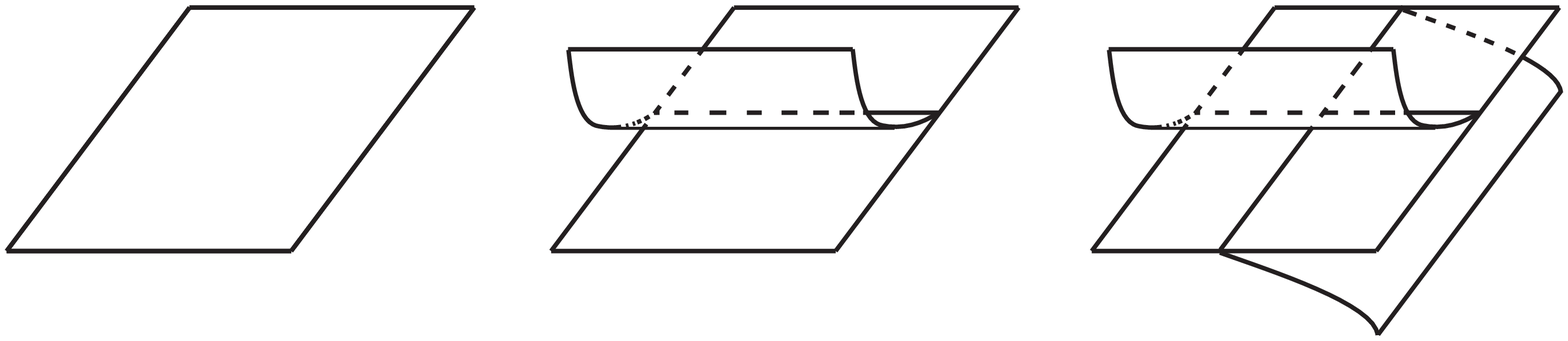}\\

  \caption{local charts of branched surface}\label{f.locbranch}
\end{center}
\end{figure}

Let $B$ be a branched surface embedded in $M$. We denote by $N(B)$ a regular neighborhood of $B$
with a (semi-)$I$-bundle  so that there exists a projection $\pi: N(B)\to B$ which collapses every interval fiber to a point.
The boundary of $N(B)$ is the union of two compact surfaces $\partial_h N(B)$ and $\partial_v N(B)$.
These two sub-surfaces can be characterized as follows.
If an interval fiber of $N(B)$ meets $\partial N(B)$,
set $P$ belongs to the intersection  of the interval fiber and $\partial N(B)$,
then $P\in \partial_h N(B)$ if the interval fiber meets
$\partial_h N(B)$ transversely and $P\in \partial_v N(B)$ otherwise.
We call $\partial_h N(B)$ the \emph{horizontal boundary} of $N(B)$ and $\partial_v N(B)$ the
 \emph{vertical boundary} of $N(B)$.
  See Figure \ref{f.regnghd} as an illustration. Note that $\partial_v N(B)$ satisfies that,
  \begin{enumerate}
    \item $\pi(\partial_v N(B))$ exactly is the branched locus of $B$;
    \item $\partial_v N(B)$ is the union of finitely many pairwise disjoint annuli when $M$ is orientable.
  \end{enumerate}
 \begin{figure}[htp]
\begin{center}
  \includegraphics[totalheight=4.4cm]{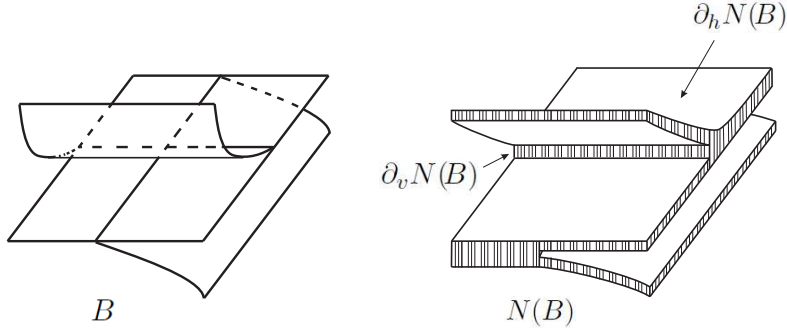}\\

  \caption{the horizontal boundary $\partial_h N(B)$ and the vertical boundary $\partial_v N(B)$ of $N(B)$}\label{f.regnghd}
\end{center}
\end{figure}

 Set $W(B)=M-\mbox{int} (N(B))$.
 Notice that $\partial W(B)=\partial N(B)$, then we can set $\partial W(B)=\partial_h W(B) \cup \partial_v W(B)$ so that $\partial_h W(B)=\partial_h N(B)$ and $\partial_v W(B)=\partial_v N(B)$.
 We say that there is a \emph{monogon} in $W(B)$ if there exists a disk $D\subset W(B)$ with $\partial D= D\cap \partial W(B)=\alpha\cup \beta$ where $\alpha \subset \partial_v N(B)$ is an interval and $\beta \subset \partial_h N(B)$.
A \emph{sink disk} is a disk $D$ in $B$ so that $\partial D$ is the union of several branched intervals and the branch direction of every smooth arc in its boundary points into the disk.

 \begin{definition}\label{d.ebra}
A branched  surface $B$ in a closed $3$-manifold $M$ is \emph{essential}  if
it satisfies the following conditions.
\begin{enumerate}
  \item $\partial_h N(B)$ is incompressible in the closure of $M- N(B)$, and
  no component of $\partial_h N(B)$ is a sphere.
  \item There is no monogon in the closure of $M- N(B)$.
  \item $B$ does not carry a torus that bounds a solid torus.
  \item $B$ has no sink discs.
\end{enumerate}
 \end{definition}

 We say a lamination $\cL$\footnote{$\cL$ possibly is a foliation.} is \emph{carried} by a branched surface $B$ if,
 after splitting finitely many leaves of $\cL$ if necessary, $\cL$ can be isotoped into $N(B)$
 so that it intersects to the $I$-bundles transversely.
 Moreover, we say that $\cL$ is \emph{fully carried} by $B$ if $\cL$ intersects every fiber of $N(B)$.

 \begin{theorem}\label{t.carry}
 A lamination $\cL$ is essential if and only if it is fully carried by an essential branched surface
 $B$ in $M$. Moreover, every essential branched surface $B$ in $M$ carries some essential lamination
 $\cL$ on $M$.
 \end{theorem}

 \begin{remark}
 Gaibai and Oertel \cite{GO} firstly defined essential branched surfaces without `the  no sink discs' condition. They further showed that the first part of Theorem  \ref{t.carry} for this kind of essential branched surfaces. Li [Li] added the `the  no sink discs' condition and showed the second part of the theorem.
 \end{remark}

\subsection{Normal form}
In \cite{Ha}, Haken showed that every  incompressible surface in a closed irreducible triangulated
$3$-manifold always can be isotopic to normal surface with respect to the triangulation.
This result is called by Haken lemma nowadays, which is a fundamental result  in the theory of $3$-manifolds.

We say that a lamination $\cL$ is \emph{normal} with respect to a triangulation $\tau$ on the background $3$-manifold,
if each plague in the intersection of  a leaf of $\cL$ with a simplex $\sigma$ of $\tau$ is a normal disk
in the sense of Haken \cite{Ha}, which is either a triangle or a quadrilateral.
In lamination case, generally one can not isotopically push the lamination in a normal position with respect to
a fixing triangulation. The phenomena illustrated in Figure 1.1  of Gabai \cite{Ga2} explains an obstruction
to isotopically push a lamination to a normal position. Nevertheless, in \cite{Bri2}, Brittenham proved an analogue of the Haken lemma for laminations: ``If a closed orientable
$3$-manifold $M$ with triangulation $\tau$ carries an essential lamination $\lambda$, then $M$ carries an essential lamination $\cL$ (possibly different with $\lambda$)  which is normal with respect to $\tau$."
The proof by Brittenham depends on a sequence of normalizing process.
Gabai \cite{Ga2} goes on to describe when and how Brittenham's normalizing process fails to be an isotopy:

\begin{theorem}[Gabai]
Let $\lambda$ be a nowhere dense essential lamination in a closed oreintable $3$-manifold $M$ with triangulation $\tau$. Then $\lambda$ can be transformed into a normal essential lamination $\mu$ by first
deleting the interior of any generalized cylindrical components and then doing one of the following operations:
\begin{enumerate}
  \item isotopy;
  \item splitting open along a finite number of leaves followed by isotopy;
  \item evacuating a taut sutured manifold $(W, \gamma)$.
\end{enumerate}
\end{theorem}

We will never use generalized cylindrical component and taut sutured manifold, so we do not define them here. An interested reader can found their definitions in
 \cite{Ga2}.
Nevertheless, we remark that one of the main observations by Schwider is that there are no evacuations or
deleted generalized cylindrical components on the decomposition of $M(r)$ which Schwider used, i.e. each of the essential laminations in $M(r)$ can be isotopically pushed to a normal position.
But the decomposition which he used is, instead of a triangulation of the background manifold $M(r)$, the union of the solid torus $V$ and a spine decomposition of
the figure-eight knot complement $N$ which firstly introduced by Thurston \cite{Thu}. Then he defined the normal position of an essential lamination
with respect to this decomposition. Naturally, the normal position should be defined with respect to both of $V$
and  a spine decomposition of $N$. In this subsection, we will introduce the first part and the second part will be introduced in Section \ref {ss.sch}.

In \cite{Bri1}, Brittenham carefully discussed a kind of good positions for the intersection of  an essential lamination and a solid torus in the background $3$-manifold. The main related result is Theorem 3.1 of \cite{Bri1}. To avoid to introduce too many new conceptions, we only introduce a simplified version of Brittenham's theorem by Schwider (Theorem II.7 of \cite{Sch}), which works  for the essential laminations on $M(r)$. Note that this is the first type of normal form
to push an essential lamination on $M(r)$ to a good position.

Let $\lambda$ be  an essential lamination on some $M(r)$ so that $\lambda$ is transverse to $\partial V$.
Set $\lambda_v = \lambda \cap V$ and $\partial \lambda_v = \lambda_v \cap T$ (recall that $T=\partial V$) which is a $1$-lamination
on $T$.

\begin{theorem}\label{t.norsoltor}
If there does not exist a circle $c$ in $\partial \lambda_v$ so that $c$ satisfies that
 \begin{enumerate}
   \item either $c$ bounds a disk plague of $\lambda$ outside  of $V$ (equivalently, in $N$);
   \item or $c$ bounds a disk plague of $\lambda$ inside  of $V$ and also bounds a disk in $T$,
 \end{enumerate}
then $\lambda_v$ is either a collection of meridian discs, or there is a standard Seifert-fibering of $V$ so that
(after isotopy) $\lambda_v$ contains a vertical sub-lamination $\lambda_v^0$ whose leaves are annuli, and each leaf of $\lambda_v -\lambda_v^0$ are non-compact, simply connected, and horizontal.
\end{theorem}

\begin{remark}
Note that as explained in \cite{Bri1}, up to isotopy, an essential lamination always  satisfies the conditions of Theorem \ref{t.norsoltor}. This means that an essential lamination $\lambda$ on $M(r)$ always can be isotopically pushed to the position so that it meets $V$ in only meridian discs, or in annuli and simply connected leaves.
\end{remark}

\subsection{Spine decomposition of the figure-eight knot complement $N$}\label{ss.spine}

We give  a hexagonal decomposition of a torus $T^2$ with four hexagons as Figure \ref{f.cellT2} shows.
\begin{figure}[htp]
\begin{center}
  \includegraphics[totalheight=3.2cm]{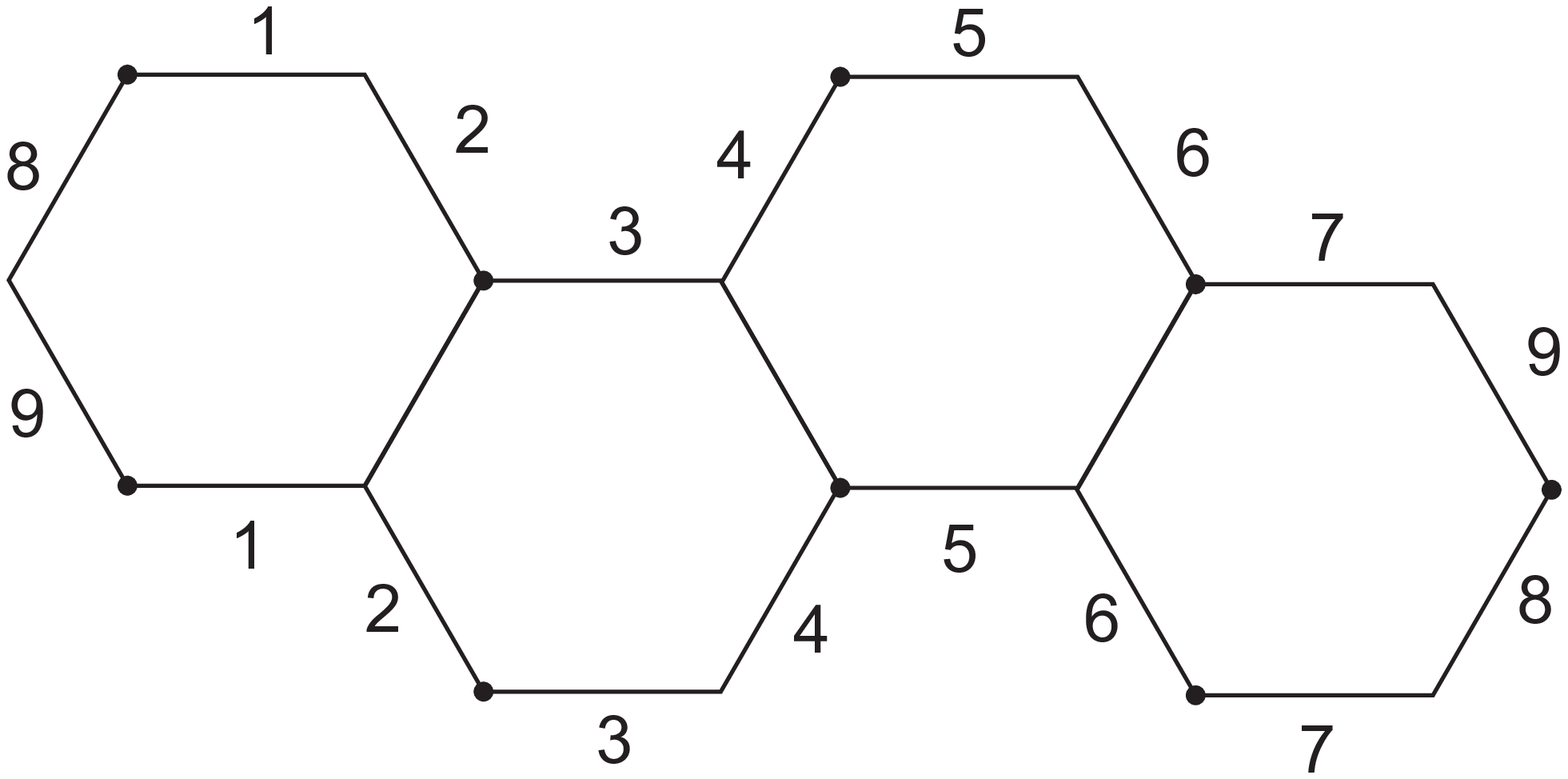}\\

  \caption{a hexagonal decomposition of a torus $T^2$}\label{f.cellT2}
\end{center}
\end{figure}

The decomposition of $T^2$ naturally induces a cell decomposition on $T^2 \times [0,1]$
with four $3$-cells  so that each cell is  associated to the product of a hexagon and an interval.
Build a degree two reflection $\tau: T^2 \times \{0\} \to T^2 \times \{0\}$
as Figure \ref{f.psref} shows. The quotient space  $\sigma = T^2 \times \{0\}/ (x,0)\sim (\tau(x),0)$ is a $2$-complex
with the set of $2$-cells $\sigma^2=\{X, Y\}$, the set of $0$-cells $\sigma^0=\{P_1, P_2\}$ and the set of $1$-cells $\sigma^1=\{a, b, c, d\}$.
\begin{figure}[htp]
\begin{center}
  \includegraphics[totalheight=4cm]{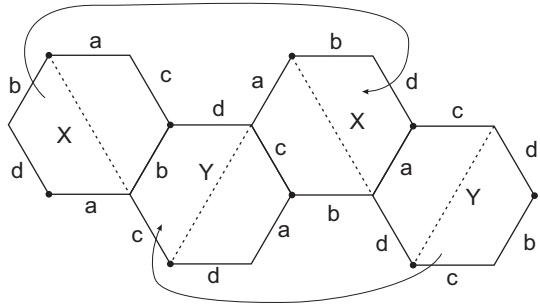}\\

  \caption{the reflection $\tau$ and the spine $\sigma$}\label{f.psref}
\end{center}
\end{figure}

In \cite{Thu}, Thurston showed that up to homeomorphism, the figure-eight knot complement $N$ can be parameterized by $T^2 \times [0,1] / (x,0)\sim (\tau(x),0)$ with $T=\partial N=T^2 \times\{1\}$.  This decomposition of $N$ is called by a \emph{spine decomposition}  of $N$, and
$\sigma$ is called by  a \emph{spine} of $N$. Roughly speaking, the spine decomposition is a singular twisted I-bundle over $\sigma$.
One can find more related information about the spine decomposition in \cite{Thu} or \cite{Sch}.

\subsection{The main strategy of Schwider's classification}\label{ss.sch}
In Thurston's classification of incompressible surfaces in $M(r)$ \cite{Thu},
he put the surface in a good position with respect to the spine $\sigma$.
To do this, first he assumes that the surface meets the $1$-skeleton of $\sigma$ transversely and minimally,
and then he isotopes the surface into normal form with respect to $\sigma$.

There are two difficulties to generalize the idea to laminations:
\begin{enumerate}
  \item there always are infinitely many leaves in a lamination so that the minimal intersection makes no sense now;
  \item the non-compactness of the leaves also makes it difficult to put a lamination in normal form.
\end{enumerate}

Now we introduce Schwider's normalization of an essential lamination $\lambda$ on $M(r)$ with respect to $N$.
The first step of  normalization  is to isotopically move $\lambda$ so that $\lambda$ is transverse to the spine $\sigma$. Here we say that $\lambda$ is \emph{transverse} to $\sigma$ if every leaf of $\lambda$ is transverse to $\sigma$ and is disjoint with the two vertices
of $\sigma$.
This always can be done since we can think that $\lambda$ is fully carried by a branched surface $B$, and then we can isotopically move $B$ so that $B$ and $\sigma$ are transverse and $B\cap \sigma$ is disjoint with the two vertices of $B$ due to the general intersection theory.  Let us be more precise.

\begin{definition}\label{d.goodp}
We say that an essential lamination $\lambda$ on $M(r)$ is in \emph{good position} with respect to $N$ if
\begin{enumerate}
  \item $\lambda$ is transverse to $\sigma$;
  \item $((\lambda\cap N)-\sigma)$ is homeomorphic to $(\lambda \cap T) \times (0,1]$, where $\lambda \cap T$ double covers $\lambda \cap \sigma$. Here recall that $T= \partial N = T^2 \times \{1\}$.
\end{enumerate}
\end{definition}

The second step of normalization will never alter the good position of $\lambda$ since each isotopy will only take leaves through $\sigma$. For the second step of normalization, Brittenham introduce a weight pair $(\cW, w)$ on each
lamination $\lambda$.

$\cW$ is defined as follows.
$\lambda^c$ is defined to be the path closure of $\lambda -\sigma$. We say that two components in $\lambda^c$ are equivalent under $\sim$ if they are isotopic
through leaves transverse to $\sigma^1$ and not meeting $\sigma^0$.
Let
$\cC= \{c\in \lambda^c: \mbox{c is compact}\}/ \sim$.
It is not difficult to show that $\cC$ is a finite set (see the proof of Lemma II.8 in \cite{Sch}).
Then we can define a weight $\cW=\sum_{[c]\in \cC} |c\cap \sigma^1|$ of $\lambda$, where
 $|c\cap\sigma^1|$ is the number of the points in $c\cap \sigma^1$.

In \cite{Ga2}, Gabai defined another weight $w$ of $\lambda$ by minimizing $|B\cap \sigma^1|$ for the branched surface set $\{B\}$ so that each $B$ is transverse to $\sigma$ and disjoint with $\sigma^0$, and carries $\lambda$. The weight pair $(\cW, w)$ is
defined by the lexigraphical order.

A \emph{short connector} is a normal arc connecting two adjacent sides of a $2$-cell of $\sigma$.
Here a normal arc is a component of the intersection of  a leaf of $\lambda$ and a hexagon of $\sigma$, which is disjoint with the vertices of $\sigma$. Schwider call a normal arc in $\sigma$ a \emph{connector}.
Two short connectors about the same vertex in $\sigma$ are \emph{adjacent}, if they meet along an edge in $\sigma$. Schwider \cite{Sch} shows that  up to isotopy, he can get rid of adjacent short connectors and
does not encrease $(\cW, w)$. From now on, we assume that $\lambda$ is in the position so that $\lambda$ minimizes $(\cW, w)$ and does not contain any short connector.

Every hexagon of $\sigma$ contains six sides. Note that two sides of a hexagon maybe associated to the same edge of $\sigma$. Under the assumptions about the position of $\lambda$,  there does not exist any connector whose two ends are in the same side of a hexagon. Therefore, there are the following three types of connectors: \emph{short connectors},
\emph{medium connectors} and \emph{long connectors}. Here, a short connector joins two adjacent sides, a medium connector
skips over one side and a long connector joins two opposite sides in the hexagon.
 Notice that $B\cap T$ double covers $B\cap \sigma$, and there does not exist any monogon in $B\cap T$
since $B$ is essential.

Depending on these restrictions, to get the essential branched surfaces list which carries all
essential laminations on $\{M(r)\}$, Schwider first classifies all of the branched $1$-manifolds $\{Q=B\cap \sigma\}$. Let $w_a$, $w_b$, $w_c$ and $w_d$ be the number of the intersection points of $Q$ (or $B$) and the respective edges. Up to symmetries of $\sigma$, there are the following four types:
\begin{enumerate}
  \item $w_c =w_d =0$;
  \item $w_a=w_c=0$;
  \item $w_c=0$;
  \item each of $w_a$, $w_b$, $w_c$ and $w_d$ is positive, which is called \emph{All Positive} by Schwider.
\end{enumerate}

Schwider classified  $11$ types of $Q$ (Chapter III in \cite{Sch}) up to symmetry by discussing the four types above.
Then he carefully constructed $B$ from each $Q$ partially depending on  Theorem \ref{t.norsoltor}
and the fact that $B_N = B \cap N$ is  a twisted $I$-bundle over $Q$.
 By consider all possible $Q$,
 he got $39$ types of essential branched surfaces (Chapter IV in \cite{Sch}) so that
 each essential lamination can be fully carried by one of them.

In the next section, we will introduce Schwider's list and the main ideas about his proof.

%% file: Schwiderlist.tex
\section{Schwider's branched surfaces}\label{s.Schwiderlist}
To understand the list of Schwider's branched surfaces, we have to introduce more notations which were introduced by Schwider in \cite{Sch}. Recall that the spine $\sigma$ is the union of two hexagons $X$ and $Y$ with two vertices and four edges. See Figure \ref{f.labhgon}, we label the sides and vertices of
$X$ and $Y$. Here $x_1, x_2$ and $x_3$ ($x\in\{a,b,c,d\}$) are associated to the same edge $x$ of $\sigma$, the vertices with black dots are associated to the vertex $P_1$ and the other vertices are associated to $P_2$.

\begin{figure}[htp]
\begin{center}
  \includegraphics[totalheight=3.8cm]{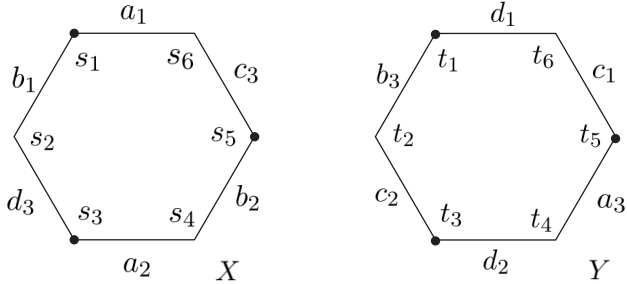}\\

  \caption{labeling sides and vertices of $X$ and $Y$}\label{f.labhgon}
\end{center}
\end{figure}

Moreover, when we discuss $B\cap \sigma$, we also use $s_i$ and $t_j$ ($i,j\in\{1,\dots,6\}$) to represent the corresponding short connectors. We can label medium connectors and long connectors as Figure \ref{f.labcon} shows. For every special connector type $u$, we denote by $|u|$ the number of connectors of type $u$.

\begin{figure}[htp]
\begin{center}
  \includegraphics[totalheight=3.8cm]{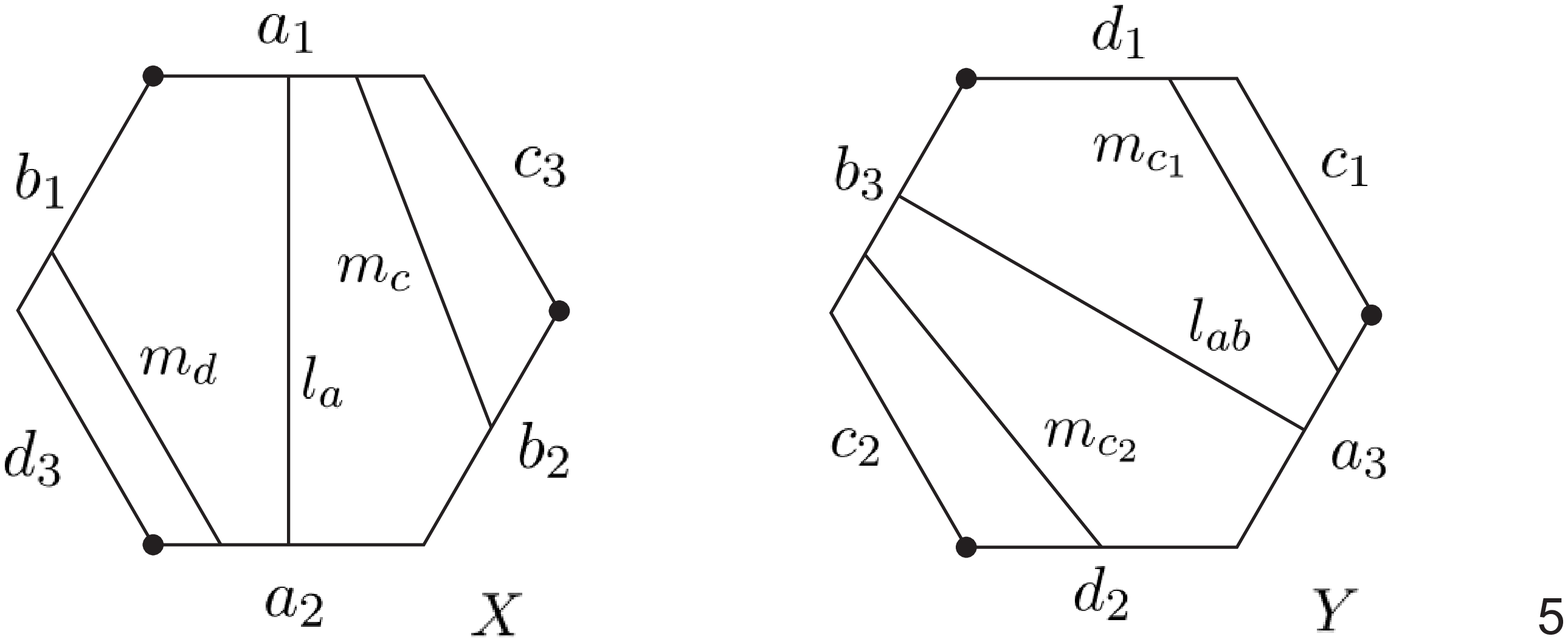}\\

  \caption{labeling some connectors}\label{f.labcon}
\end{center}
\end{figure}

 Schwider's strategy  contains the following steps more or less. Firstly, for each type in the last part of the last section, he can classify the corresponding $\{Q\}$. Notice that he proved that the first three classes were enough to provide all of the branched surfaces in question.
Secondly,  by using double cover of $Q=B\cap \sigma$, he can describe $B \cap T$.
 Finally, by using Theorem \ref{t.norsoltor} and some other techniques, he further can understand the plaques in $B^v$, and the gluing between $B^n$ and $B^v$. Here and below, $B^n$ and $B^v$ are denoted by  $B\cap N$ and $B\cap V$ respectively.

\subsection{$w_c =w_d =0$}
In this case, by some careful analysis, Schwider got that
$Q=B\cap \sigma$ has five
 possibilities: $Q_1$, $Q_2$, $Q_3$, $Q_4$ and $Q_5$, which are shown in  Figure \ref{f.Q1-5}.
 One can find more related information in Section 3.2 of \cite{Sch}.
 \begin{figure}[htp]
\begin{center}
  \includegraphics[totalheight=6cm]{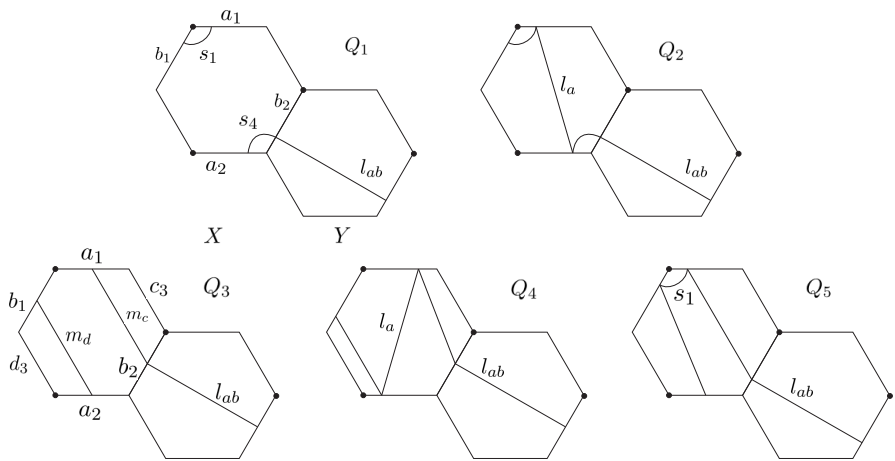}\\
 \caption{$Q_1$, $Q_2$, $Q_3$, $Q_4$ and $Q_5$}\label{f.Q1-5}
\end{center}
\end{figure}

 Note that in the case of $Q_1$ and $Q_2$, $|s_1|, |s_4|>0$, and in the case of $Q_3$, $Q_4$ and $Q_5$, $|m_c|, |m_d|>0$.
 We will describe the  set of  branched surfaces $\{B\}$ associated to $Q_1, \dots, Q_5$ one by one.

 \subsubsection{$Q_1$}\label{sss.Q1}
 $B_1^n$ is a fiber punctured  torus of $N$.
See Figure \ref{f.B1NcapT} for $\partial B_1^n=B_1^n \cap T$.
 \begin{figure}[htp]
\begin{center}
  \includegraphics[totalheight=3.2cm]{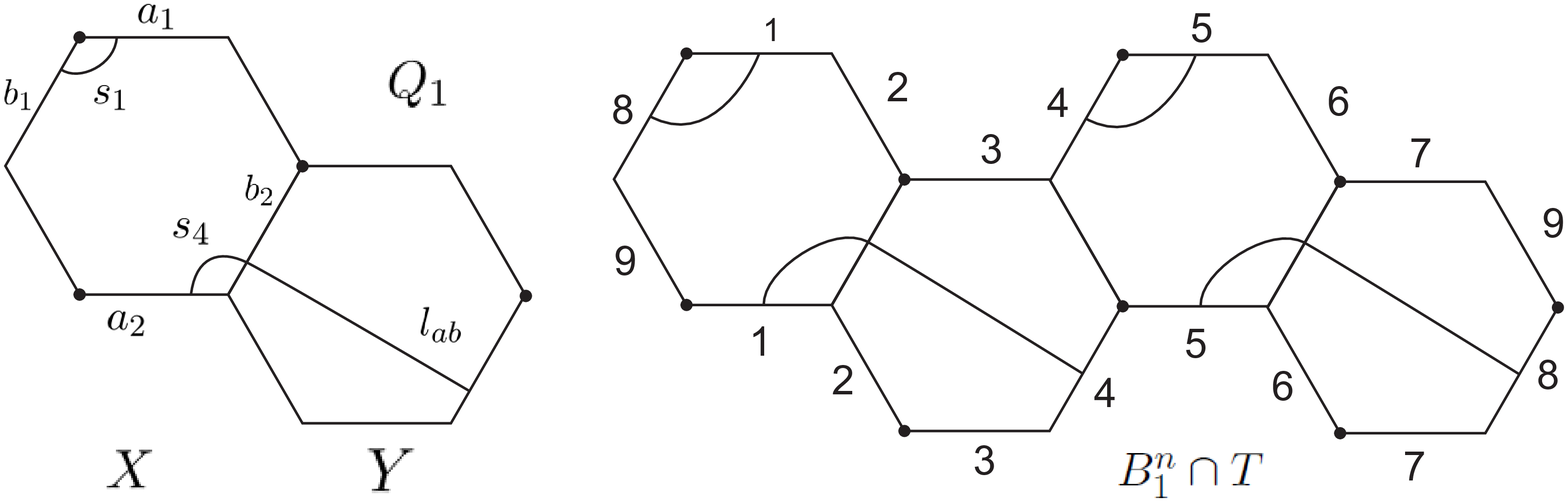}\\
  \caption{$B_1^n \cap T$}\label{f.B1NcapT}
\end{center}
\end{figure}

Schwider further showed that there is no annulus plagues in $V$ in this case. Therefore, by  Theorem \ref{t.norsoltor}, there is a unique branched surface $B_1$ associated to $B_1^n$ which in fact is homeomorphic to a fiber torus on the  sol-manifold $M(0)$. Then the set of the corresponding background manifolds  is $\{M(0)\}$.

\subsubsection{$Q_2$}
See Figure \ref{f.B2NcapT} for $\partial B_2^n=B_2^n \cap T$.
$\omega$, $\mu$ and $\nu$  code the weights of a loop\footnote{Here and below, a loop means a simple closed curve.} carried by $\partial B_2^n$ on the corresponding edges.
\begin{figure}[htp]
\begin{center}
  \includegraphics[totalheight=3.2cm]{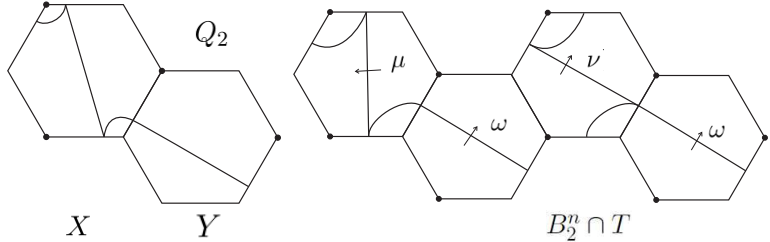}\\
  \caption{$B_2^n \cap T$}\label{f.B2NcapT}
\end{center}
\end{figure}

Any loop $c$ carried by $\partial B_2^n$ must
satisfies the switch equation induced by $\partial B_2^n$.
Then
$c$ has slope
$\frac{\mu-\nu}{\omega}$ ($\omega >0, \omega > \nu$). So $\lambda_N$ can contain loops of any slope.
For  a fixing slope, Schwider further showed that there is no annulus plagues in $V$ in this case. Therefore, by  Theorem \ref{t.norsoltor}, there is  a unique branched surface $B_2$ by gluing $B_2^n$ and a disk
along the loop on $T$ with slope $r$.
See Figure \ref{f.locatt1} for a local picture for the gluing between $B_2^n$ and $B_2^v$. Schwider proved that $B_2$ is an essential
branched surface in $M(r)$.  Since a loop $c$ carried by $Q_2$ can be any slope, then the set of the corresponding background manifolds is  $\{M(r)\mid r\in\QQ\}$.
 \begin{figure}[htp]
\begin{center}
  \includegraphics[totalheight=4.5cm]{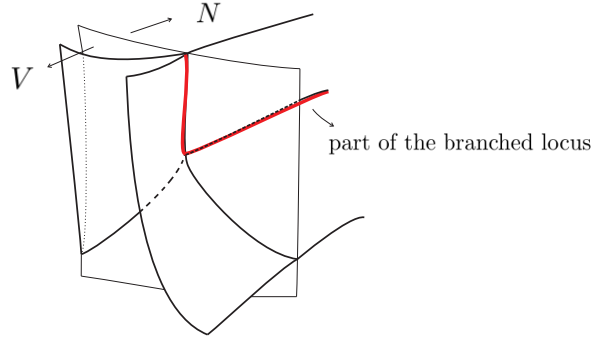}\\
  \caption{local picture of the gluing between $B_2^n$ and $B_2^v$}\label{f.locatt1}
\end{center}
\end{figure}

\subsubsection{$Q_3$}\label{sss.Q3}
 See Figure \ref{f.B3NcapT} for $\partial B_3^n=B_3^n \cap T$. $B_3^n$ is homeomorphic to a once-punctured Klein
 bottle, and the loops carried by $\partial B_3^n$
  Similar to the cases of $Q_1$ and $Q_2$, there is no annulus plagues in $V$ and there is a unique branched surface $B_3$  in this case, which in fact is a Klein bottle.
  Moreover, the set of the corresponding background manifolds  is $\{M(4)\}$.
  \begin{figure}[htp]
\begin{center}
  \includegraphics[totalheight=3.2cm]{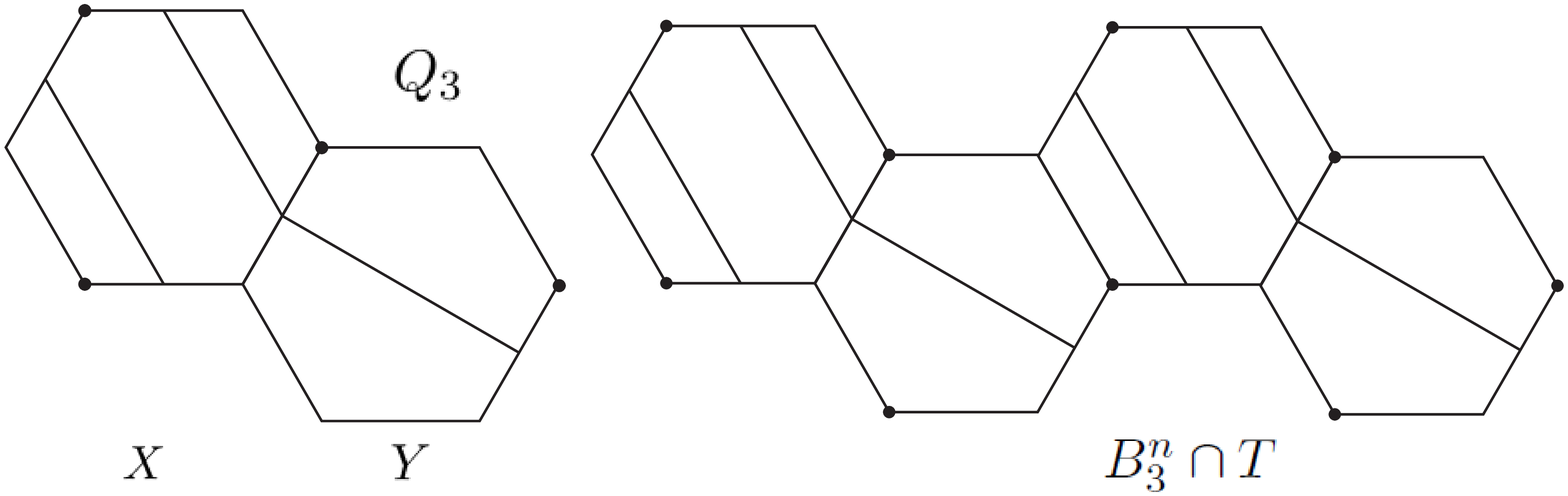}\\
  \caption{$B_3^n \cap T$}\label{f.B3NcapT}
\end{center}
\end{figure}

\subsubsection{$Q_4$}
  See Figure \ref{f.B4NcapT} for $\partial B_4^n=B_4^n \cap T$. $\omega$, $\mu$ and $\nu$  code the weights of a loop on the corresponding edges.
  Any loop carried by $B_4^n \cap T$ must has slope
$r=3+\frac{\mu + \nu}{\omega}$ ($\omega >0$, $\mu> 0$, $\omega >\nu> 0$). Similarly, Schwider showed that there is no annulus plagues in $V$ in this case. Therefore, by  Theorem \ref{t.norsoltor}, there is  a unique branched surface $B_4$ by gluing $B_4^n$ and a disk
along the loop carried by $\partial B_4^n$ on $T$ with slope $r=3+\frac{\mu + \nu}{\omega}$. In this case, the set of the corresponding background manifolds  is  $\{M(r)\mid r>3\}$.

  \begin{figure}[htp]
\begin{center}
  \includegraphics[totalheight=3.2cm]{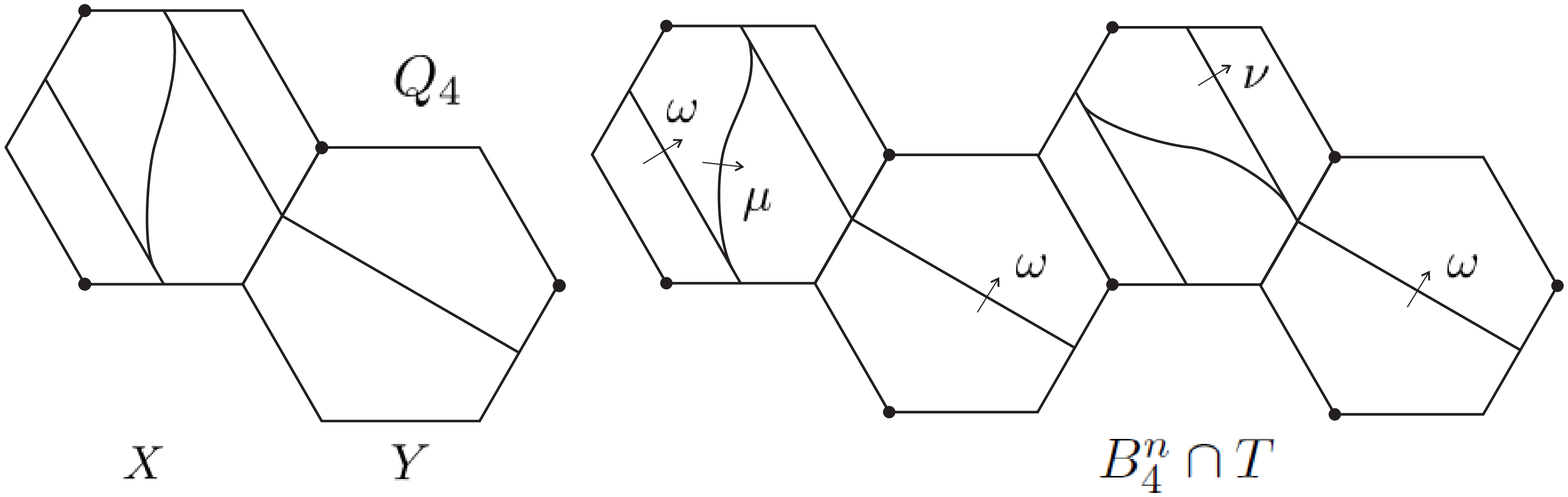}\\
  \caption{$B_4^n \cap T$}\label{f.B4NcapT}
\end{center}
\end{figure}

\subsubsection{$Q_5$}
See Figure \ref{f.B5NcapT} for $\partial B_5^n = B_5^n\cap T$. $h$, $i$ and $g$ code the corresponding edges, and
$\mu$ and $\nu$  code the weights of a loop carried by $\partial B_5^n$ on $h$ and $i$ respectively.
  \begin{figure}[htp]
\begin{center}
  \includegraphics[totalheight=3.2cm]{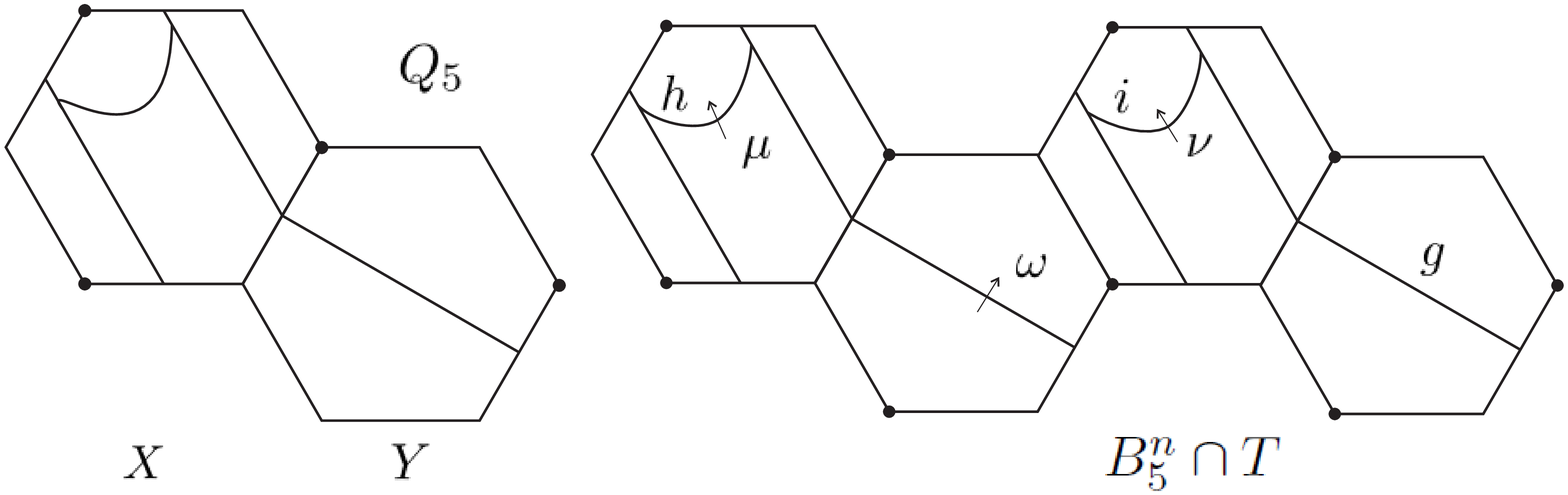}\\
  \caption{$B_5^n \cap T$}\label{f.B5NcapT}
\end{center}
\end{figure}

 Since the corresponding switch equation will induce that $\mu=\nu=0$, $\partial B_5^n$ only carries loops of slope $4$, and it must carries non-compact curves by $h$ and $i$.
 By Theorem \ref{t.norsoltor}, $\lambda_N$ does not contain  meridian discs, but contains
 annular leaves. To get $B_5$
 from $B_5^n$, we need two steps of gluing sectors.
 \begin{enumerate}
   \item Firstly, we get $B_5'$ by attaching an annulus $A$ in $V$ so that both of the two ending loops are glued to $g$. The union of $A$ and $T$ cuts $V$ to two components: the \emph{occupied} component and the \emph{vacant} component. Here a component is occupied (vacant, resp.) if it contains some (does not contain any, resp.) horizontal leaves.
   \item Then we get $B_5$ by attaching two disc sectors whose interiors are in the occupied component so that each disc sector has one are glued to either $h$ or $i$ and
       the complementary arc meets $A$ in a branch curve (see Figure \ref{f.attdisk} as an illustration). Note that when the manifold $M(r)$ is fixed, there is a unique way
       that these disc sectors meet $A$ in a branch. Therefore, $B_5$ is unique when $M(r)$ is fixed.
 \end{enumerate}
Further notice that to avoid an end-compression, the slope-$r$ loop of
 the background manifold $M(r)$ should intersects the slope-$4$ loop at least twice.
 This means that
 the set of the corresponding background manifolds  is
$\{M(r)=M(\frac{q}{p})\mid |4p-q|\geq 2\}$.
  \begin{figure}[htp]
\begin{center}
  \includegraphics[totalheight=4.5cm]{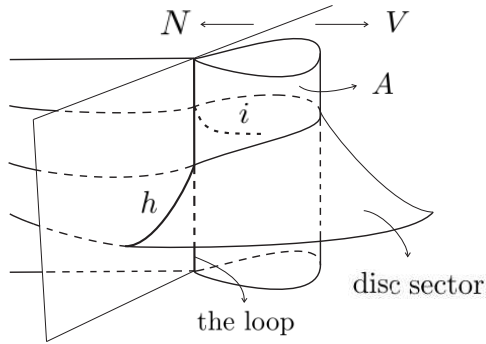}\\
  \caption{attaching a disc sector}\label{f.attdisk}
\end{center}
\end{figure}

\subsection{$w_a =w_c =0$}
In this case, $Q$ has two types: $Q_6$ and $Q_7$, see the left parts of Figure \ref{f.B6NcapT} and Figure \ref{f.B7NcapT}. More details can be found in Section 3.3 of \cite{Sch}.

\subsubsection{$Q_6$}\label{ss.Q6}
See Figure \ref{f.B6NcapT} for $\partial B_6^n = B_6^n\cap T$. It is easy to know from the related switch equation that $\partial B_6^n$ can only carry loops with
slope $\infty$.  Moreover, these loops only can be carried by $g$ or $h$, and all other arcs in
$\partial B_6^n$ only can carry  non-compact curves. By Theorem \ref{t.norsoltor},
$\lambda_v$ does not contain meridian discs. Here $\lambda_v =\lambda \cap V$ where $\lambda$ is an essential lamination on $M(r)$ so that $\lambda \cap N$ is fully carried by $B_6^n$.
  Essentially due to the facts that there is no monogon
disk for $B$ and the number of  fibers in $V$ is no more than $1$,
Schwider shows that $\lambda_v$ contains some vertical annuli so that each of them satisfies that one boundary component is in
$N(g)$ and the other boundary component is in $N(h)$. This fact can ensure us to get the first branched surface $B_6$ associated to $Q_6$, which is obtained by two steps:
\begin{enumerate}
  \item to add one annular sector with two boundary components are in $N(g)$ and $N(h)$ respectively;
  \item to add several disks through the other arcs of $\partial B_6^n$.\footnote{When $M(r)$ is fixed, the way of attaching disk sectors is unique. Therefore,
      for simplicity, below when we introduce a new branched surface, sometimes we will omit talking about the attaching disks surgery.}
\end{enumerate}
\begin{figure}[htp]
\begin{center}
  \includegraphics[totalheight=3.2cm]{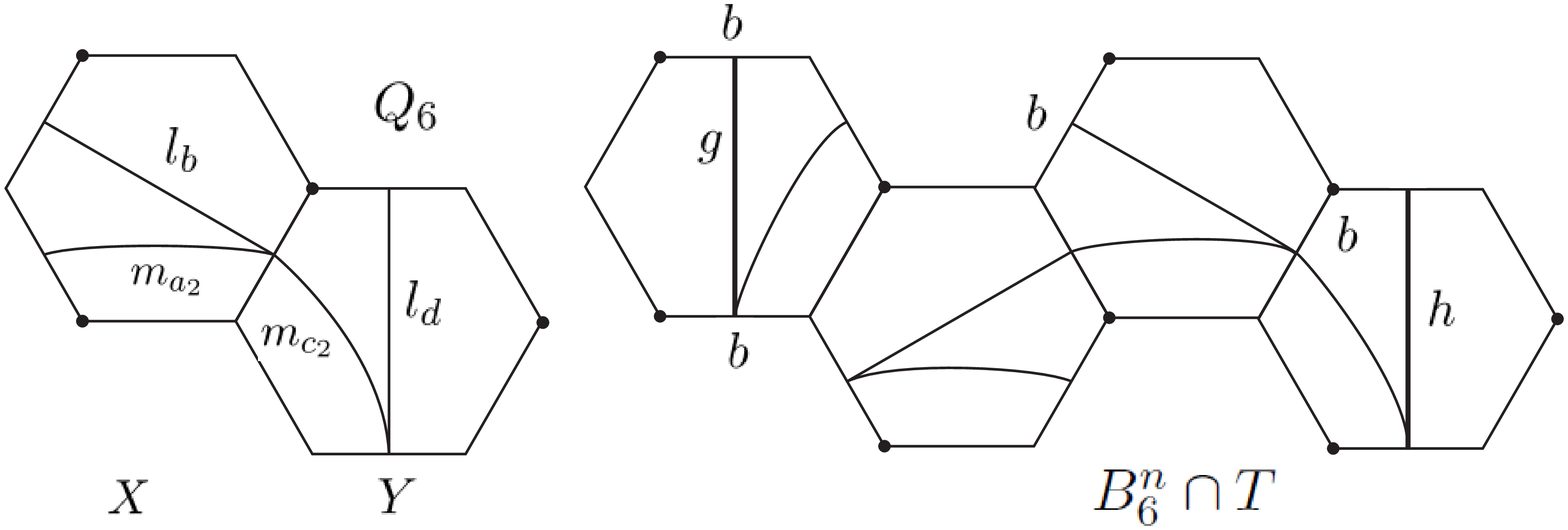}\\
 \caption{$B_6^n\cap T$}\label{f.B6NcapT}
\end{center}
\end{figure}
Notice that the complement of $B_6$ in $M(r)$ is  a solid torus with two slope-$\infty$ cusps,
therefore, in this case $M(r)$ satisfies that $r$ can be any rational number.

There are more essential branched surfaces associated to $Q_6$. To introduce them, we start by defining two types of annular sectors:
\begin{enumerate}
  \item \emph{type I sector}: an annular sector in $B\cap V$ with both ends meeting the same closed branch curve;
  \item \emph{type II sector}: an annular sector in $B\cap V$ with two ends meeting two different closed branch curves.
\end{enumerate}
Every other branched surface associated to $Q_6$ can be obtained by adding either a type I sector or a type II sector to $B_6$, which is named by $B_6^{I}$ or $B_6^{II}$ respectively. Note that
there are two types of $B_6^{I}$ which depend on   the two ends of the type I sector meeting $g$ or $h$, and there is a unique branched surface with type $B_6^{II}$, which can be obtained by gluing
a type II sector with two ends meeting $g$ and $h$.\footnote{$B_6^{II}$ abstractly can be obtained by splitting the annular sector of $B_6$ without considering embedding.}

In the case of $B(6)$, the set of the corresponding background manifolds is  $\{M(r)\mid r\in\QQ\}$.
In the case of each branched surface extended from $B_6$, the set of the corresponding background manifolds is  $\{M(r)=M(\frac{q}{p})\mid p>1\}$. Furthermore, the  solid torus is either bounded by two type II sectors,
or bounded by a type I sector. Here we say that a solid torus $V_0 \subset V$ in $M(r)=M(\frac{q}{p})$ ($p>1$) is  \emph{exceptional} if a core $c(V_0)$ of $V_0$ is isotopic to a core $c(V)$ of $V$ in $V$.
We remind a reader that this remark  also works in some similar cases below.

\subsubsection{$Q_7$}\label{ss.Q7}
Now let us turn to consider $Q_7$. See Figure \ref{f.B7NcapT} for $\partial B_7^n = B_7^n\cap T$. Similarly, $\partial B_7^n$ can only carry loops with slope $\infty$. which meets $f$, $g$ or $h$.
Schwider showed that $f$ and $h$ must meet some type II sectors (Lemma IV.21 in \cite{Sch}).
There are $20$ types of essential branched surfaces associated to $Q_7$, which can be divided into three classes.
\begin{figure}[htp]
\begin{center}
  \includegraphics[totalheight=3.2cm]{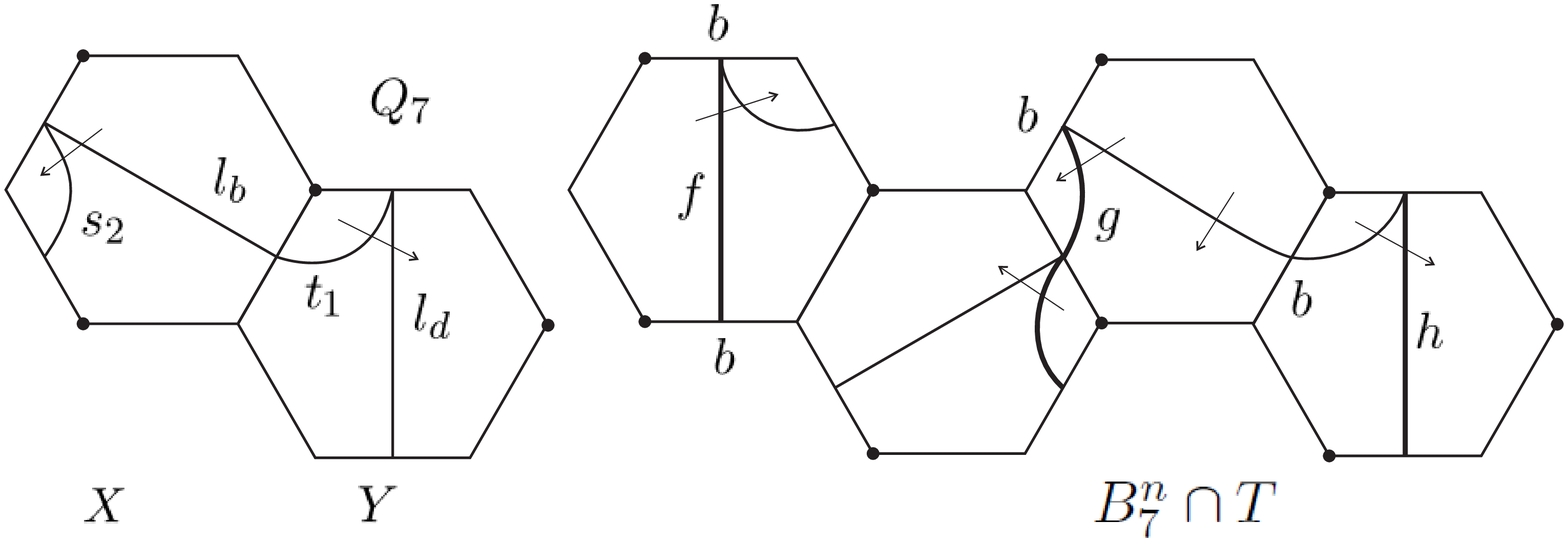}\\
 \caption{$B_7^n\cap T$}\label{f.B7NcapT}
\end{center}
\end{figure}

Each branched surface in the first class is obtained by a basic branched surface $B_7$. $B_7$
can be obtained by
 attaching two annuli sectors to $Q_7$ so that  one of them meets $f$ and $g$, and the other meets $g$ and $h$. There are $13$ essential  branched surfaces generated by $B_7$ which can be divided into the following six types.
\begin{itemize}
  \item The first type is $B_7$ itself.
  \item The second type, namely $B_7^I$, contains three branched surfaces, each one is obtained by attaching a type $I$
sector to $B_7$ meeting one of $f,g$ and $h$.
  \item The third type, namely $B_7^{II}$, contains two branched surfaces, each one is
obtained by attaching a type $II$ sector to $B_7$ meeting $f$ and $g$ or $h$ and $g$.
  \item The fourth type contains a unique branched surface $R_7$, which is obtained by gluing
a type II sector to $B_7$ meeting $f$ and $h$.
  \item The fifth type, namely $R_7^I$, contains three branched surfaces, each one is obtained by attaching a type $I$
sector to $R_7$ meeting one of $f,g$ and $h$.
  \item The sixth type, namely $R_7^{II}$, contains three branched surfaces, each one is obtained by attaching a type II
sector to $R_7$ meeting one of $f,g$, $g,h$ and $h,f$.
\end{itemize}
In the cases of $B_7$ and $R_7$, the set of the corresponding background manifolds is  $\{M(r)\mid r\in\QQ\}$.
In each of the other cases,  the set of the corresponding background manifolds is  $\{M(r)=M(\frac{q}{p})\mid p>1\}$.
Furthermore, the  solid torus is either bounded by two type II sectors,
or bounded by a type I sector.

The second class contains a unique branched surface, namely $B_7^*$, which is obtained by
 attaching a type I sector meeting $g$, and a type II sector meeting $f$ and $h$.

There are $6$ branched surfaces in the third class, for simplicity, each one is
named by $B_7^{**}$.  Each $B_7^{**}$ is obtained by gluing two type II sectors
so that one of them meets $f,h$ and the other one meets $f,g$ or $g,h$, and gluing a type I sector
meeting one of $f,g$ and $h$. Similarly, the set of the corresponding background manifolds is  $\{M(r)=M(\frac{q}{p})\mid p>1\}$. The  solid torus is bounded by the type I sector.

\subsection{$w_c =0$}
In this case, there are four types of $Q$: $Q_8$, $Q_9$, $Q_{10}$ and $Q_{11}$, which are shown in the left parts of Figure \ref{f.B8NcapT},  \ref{f.B9NcapT},  \ref{f.B10NcapT} and  \ref{f.B11NcapT}. More details can be found in Section 3.4 of \cite{Sch}.

\subsubsection{$Q_8$}\label{ss.Q8}
In the case of $Q_8$, $\partial B_8^n = B_8^n \cap T$ only carries loops with slope $\infty$
which only can be carried by $f$, $g$ and $h$, see Figure \ref{f.B8NcapT}.
\begin{figure}[htp]
\begin{center}
  \includegraphics[totalheight=3.7cm]{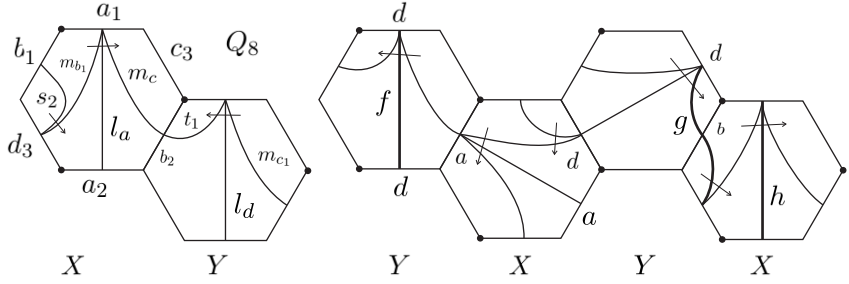}\\
 \caption{$B_8^n\cap T$}\label{f.B8NcapT}
\end{center}
\end{figure}
Schwider showed that a type II sector must meet $g$ (Lemma IV.29 in \cite{Sch}).
The basic branched surface $B_8$ consists of two type II sectors so that one meets $f$ and $g$,
and the other  meets $f$ and $h$.
By abstractly splitting one of the two type II sectors, we have two types of $B_8^{II}$.
There is another branched surface associated to $Q_8$, namely $B_8^{III}$, which is obtained by
 attaching a type II sector to $B_8$ meeting $g$ and $h$,
and a type I sector meeting $g$.

In the case of $B_8$,
the set of the corresponding background  manifolds is $\{M(r)\mid r\in \QQ\}$.
In each of the other cases,  the set of the corresponding background  manifolds is
$\{M(r)=M(\frac{q}{p})\mid p>1\}$.

\subsubsection{$Q_9$}\label{ss.Q9}
In the case of $Q_9$, see Figure \ref{f.B9NcapT} for $\partial B_9^n = B_9^n\cap T$. We use $e$, $f$, $g$, $h$, $i$ and $j$ to code
both of the edges and
the weight of the corresponding edges.
\begin{figure}[htp]
\begin{center}
  \includegraphics[totalheight=3.7cm]{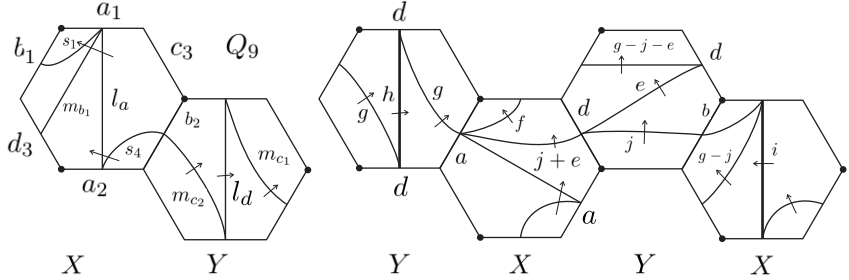}\\
 \caption{$B_9^n\cap T$}\label{f.B9NcapT}
\end{center}
\end{figure}
We divide the discussions to two cases
depending on that if
 the weight of a  loop carried by  $\partial B_9^n$ is $0$  on $g$. If $g=0$,
 by gluing a type II sector meeting $h$ and $i$, we get the first basic branched surface $B_9$.
By abstractly splitting the glued type II sector in $B_9$, we get $B_9^{II}$.

If $g>0$, Schwider showed that $h>0$, $M(r)$ satisfies that $r=\frac{q}{p}$ with slope $\frac{g+h-f-e-i}{g}$, and there must be only some meridian discs in $\lambda_v$.
Here $\lambda_v$ is the intersection of the lamination with the solid torus $V$.
We call by $B_9^M$ this new branched surface.

In each case of $B_9$ and $B_9^M$,
the set of the corresponding background  manifolds is $\{M(r)\mid r\in \QQ\}$.
In the case of $B_9^{II}$,  the set of the corresponding background  manifolds is
$\{M(r)=M(\frac{q}{p})\mid p>1\}$.

\subsubsection{$Q_{10}$}
In the case of $Q_{10}$, see Figure \ref{f.B10NcapT} for $\partial B_{10}^n = B_{10}^n\cap T$.
\begin{figure}[htp]
\begin{center}
  \includegraphics[totalheight=3.7cm]{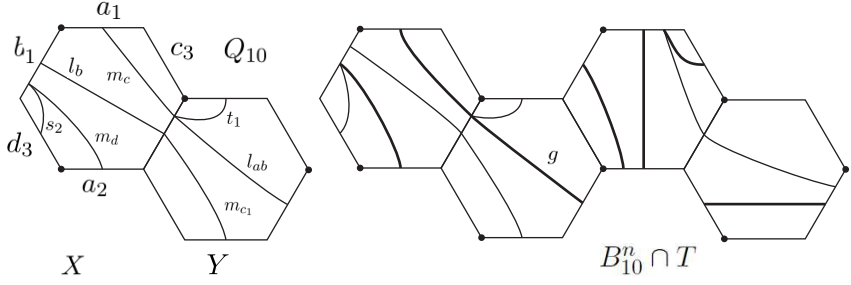}\\
 \caption{$B_{10}^n\cap T$}\label{f.B10NcapT}
\end{center}
\end{figure}
By computing the switch equation, every loop must be carried by $g$.
$B_{10}$ can be obtained by attaching a type I sector to $B_{10}^n$ along
$g$. Here $g$ is a  loop on $T$ with slope $4$.
The set of the corresponding background  manifolds is
$\{M(r)=M(\frac{q}{p})\mid |4p-q|>1\}$.

\subsubsection{$Q_{11}$}
In the case of $Q_{11}$, see Figure \ref{f.B11NcapT} for $\partial B_{11}^n = B_{11}^n\cap T$.
\begin{figure}[htp]
\begin{center}
  \includegraphics[totalheight=3.7cm]{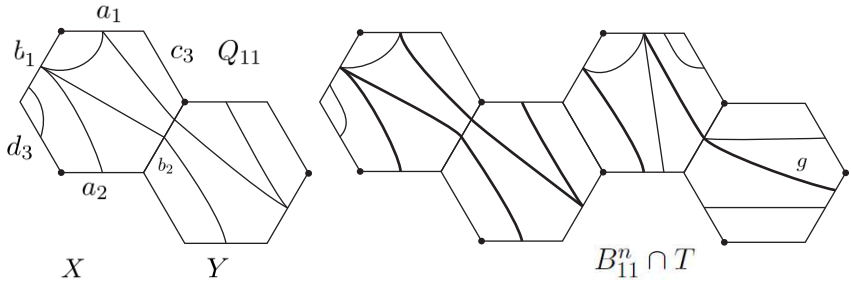}\\
 \caption{$B_{11}^n\cap T$}\label{f.B11NcapT}
\end{center}
\end{figure}
Similar to the case of $Q_{10}$, every loop must be carried by $g$, which now is a sub-train track of $\partial B_{11}^n$. But one can easily observe that every loop carried by $g$ also should be slope $4$.
The corresponding branched surface $B_{11}$ can be obtained by
attaching a type I sector to $B_{11}^n$ along
$g$. The set of the corresponding background  manifolds is
$\{M(r)=M(\frac{q}{p})\mid |4p-q|>1\}$.

%% file: trori.tex
\section{Transverse orientation of branched surface}\label{s.trori}
Let $B$ be a branched surface in a closed orientable $3$-manifold $M$, and $N(B)$ be a regular neighborhood of $B$ with a (semi-)I-bundle induced by a projection $\pi: N(B)\to B$.
We say that a loop $c\subset B$ is \emph{legal} if for every intersectional point $P$ of $c$ and the branched locus of $B$, $c$ crosses the branched locus either from two folds side to one fold side or from one fold side to two folds side.
We say that $B$ is \emph{transversely orientable} if the (semi-)I-bundle on $N(B)$ is orientable, equivalently, for every legal loop  $c\subset B$, the (semi-)I-bundle restricted on  $c$
is orientable.

Transverse orientation of branched surface is useful for us because of the following lemma.
\begin{lemma}\label{l.triori}
Let $B$ be a transversely orientable branched surface in a closed orientable $3$-manifold $M$,
and $\cL$ be a lamination carried by $B$. Then $\cL$ is orientable and transversely orientable.
\end{lemma}
\begin{proof}
Since $M$ is orientable, $\cL$ is orientable if and only if $\cL$ is transversely orientable.
Therefore, we only need to show that $\cL$ is transversely orientable.
Let $l$ be a leaf of $\cL$ and $c$ be a closed curve in $l$. Up to isotopy of $c$ in $l$,
we can assume that $c$ satisfies that the loop $\pi(c)$ in $B$ is legal. By the condition
that $B$ is transversely orientable, the  (semi-)I-bundle on $\pi(c)$ induced by the  (semi-)I-bundle on $N(B)$ is orientable. Further notice that the  (semi-)I-bundle on $N(B)$
is transverse to $\cL$. Then the holonomy of $\cL$ along  $c$ induced by   the  (semi-)I-bundle on $N(B)$ is orientable.  Thereore, $\cL$ is transversely orientable.
\end{proof}

For example, recall that (Section \ref{sss.Q1}) $B_3$ is an embedded Klein bottle, therefore, $B_3$ is not transversely orientable. In the following proposition, we will check that
some branched surfaces in Schwider's list are transversely orientable.

\begin{proposition}\label{p.triori}
Each of $B_6$, $B_7$, $B_8$, $B_9$ is transversely orientable in every corresponding background manifold $M(r)$.
\end{proposition}
\begin{proof}
The right part of Figure \ref{f.B6NcapTor} endows a transverse orientation on the branched $1$-manifold $B_6^n \cap T$ in $T$, which can induce a transverse orientation on $B_6^n \cap \sigma$ in $\sigma$. The induced transverse orientation on $B_6^n \cap \sigma$ can be found in
the left part of Figure \ref{f.B6NcapTor}. By the spine structure of $N$ (see Section \ref{ss.spine}), these transverse orientations naturally induce a transverse orientation on $B_6^n$ on $N$.
\begin{figure}[htp]
\begin{center}
  \includegraphics[totalheight=3.3cm]{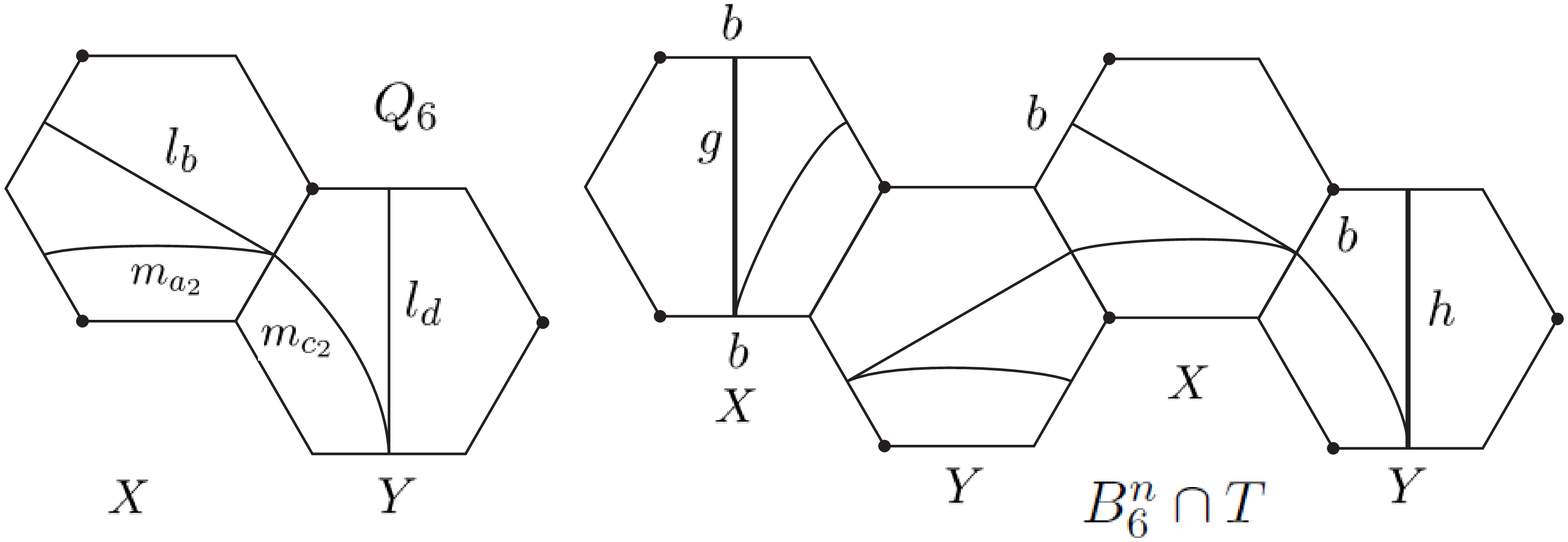}\\
 \caption{transverse orientation on $B_6^n$}\label{f.B6NcapTor}
\end{center}
\end{figure}

Recall that $B_6$ can be obtained from $B_6^n$ by gluing:
\begin{enumerate}
  \item an annular sector $P$ meeting $g$ and $h$;
  \item and several disk sectors meeting the other arcs of $\partial B_6^n$.
\end{enumerate}
By observing the right part of Figure \ref{f.B6NcapTor}, one can easily get that the
transverse orientation on $B_6^n \cap T$ can be extended to $P$.
Further notice that each branched circle in $B_6$ is
an extension of several branched intervals of $B_6^n$, and  recall that the
transverse orientation on $B_6^n \cap T$ can be extended to $P$, the transverse orientation on $B_6^n \cap T$ can be naturally extended to the disk sectors. Therefore, the
transverse orientation on $B_6^n \cap T$ can be extended to $B_6$. One also can  read  page $88$ and page $89$ of \cite{Sch} for a slightly different explanation about the transverse orientation on $B_6$.

One can similarly check the existence of transverse orientations on $B_7$, $B_8$ and $B_9$
by using the transverse orientations on the branched $1$-manifolds in
Figure \ref{f.B7NcapT}, \ref{f.B8NcapT} and \ref{f.B9NcapT}.
 \end{proof}

 The following corollary is a direct consequence of Lemma \ref{l.triori} and Proposition \ref{p.triori}.

 \begin{corollary}\label{c.B6789nocarry}
 Let $\cL$ be an essential lamination on some $M(r)$, which is  carried by one of $B_6$, $B_7$,
 $B_8$ and $B_9$. Then $\cL$ is transversely orientable.
 \end{corollary}
 \begin{proof}
 Set $\cL$ is carried by $B$ where $B\in \{B_6, B_7, B_8, B_9\}$.
 We assume that $\cL$ is not transversely orientable.
 By Lemma \ref{l.triori}, $B$ must be not transversely orientable.
 This conflicts to the conclusion of   Proposition \ref{p.triori}.
 \end{proof}

 We remark that Corollary \ref{c.B6789nocarry} will be very useful during the proof of our main result: Theorem \ref{t.main}.

%% file: Anosov.tex
\section{Anosov flows and Anosov laminations}\label{s.Anosov}
In this section, we will firstly introduce some  properties
 about Anosov foliations and Anosov laminations,
then we will  list some  useful facts about three dimensional Anosov flows.
Finally,
 we will build a useful criterion to decide when the Anosov flow associated to an Anosov foliation/lamination on
$M(r)$ ($r\in \ZZ$) is topologically equivalent to $X_t^r$,
and discuss orientations about three dimensional Anosov flows.
 Notice that we omit some basic definitions and properties.
 A careful reader can find some related information in \cite{Fen1}, \cite{BBY} and \cite{Bart}.

Let $X_t$ be an Anosov flow on a closed $3$-manifold $M$. Recall that we denote by \emph{Anosov foliation} the foliation formed by the union of weak stable/unstable manifolds of $X_t$, and  \emph{Anosov lamination} the lamination obtained by splitting finitely many leaves of an Anosov foliation. Here `splitting' is a standard surgery in foliation theory, whose serious definition can be found in Example 4.14 of Calegari's book \cite{Cal}.  Notice the fact that a weak stable/unstable manifold of $X_t$ is one of
an immersed open annulus, an immersed open Mobius band and an immersed plane. Then a complement connected component of an Anosov lamination can be endowed with one of  the following three types of structures:
\begin{enumerate}
  \item $I$-bundle over an open annulus;
  \item twisted $I$-bundle over an open Mobius band;
  \item $I$-bundle over a plane.
\end{enumerate}
This fact implies the following elementary but useful proposition.

\begin{proposition}\label{p.Anobracomp}
Let $B$ be an essential branched surface on a closed $3$-manifold $M$.
If $B$ fully carries an Anosov lamination on $M$,
then every connected component $W_0$ of the  manifold $W(B)=M-\mbox{int} (N(B))$ should carries an $I$-bundle (regular or twisted)  coherent to $\partial_v W(B)$, i.e the $I$-bundle can be extended to an $I$-bundle on $\partial_v W(B)$. Moreover, $W_0$ should be one of the following three types:
\begin{enumerate}
  \item $W_0$ is homeomorphic to a solid torus parameterized by
 $A= S^1 \times [-1,1] \times [0,1]$, and $\partial_v W_0$ is the union of two annuli with the induced parameters $S^1 \times \{-1\} \times [0,1]$ and  $S^1 \times \{1\} \times [0,1]$;
  \item $W_0$ is homeomorphic to a solid torus, and  $\partial_v W_0$ is an annulus
  so that a core of $\partial_v W_0$ runs along a core of $W_0$ twice;
  \item $W_0$ is homeomorphic to a three ball parameterized by
 $D\times [0,1]$ where $D$ is a disc, and $\partial_v W_0 = \partial D \times [0,1]$.
\end{enumerate}
\end{proposition}

In the proof of Theorem \ref{t.main}, we will also use the following fundamental result
due to Fenley
(Theorem 1.1 and Theorem 1.2 of \cite{Fen2}), which concerns
homotopy information about periodic orbits of three dimensional Anosov flows.

\begin{theorem}\label{t.fenley}
Let $X_t$ be an Anosov flow on a closed $3$-manifold $M$. If $\omega$ is a periodic orbit
of $X_t$ and $\omega$ is freely homotopic to $a^k$ ($k\in \ZZ$) where $a$ is a closed
curve in $M$, then $1\leq |k| \leq 2$.
Moreover, if there exists a periodic orbit $\omega$ of $X_t$ so that $\omega$ is freely homotopic to either $a^2$ or $a^{-2}$,
then either the stable foliation or the unstable foliation of $X_t$ is not transversely orientable.
\end{theorem}

The following classical result is a special case of Theorem B of Plante \cite{Pl}.

\begin{theorem}\label{t.plante}
Let $M$ be a closed solvable $3$-manifold and $X_t$ be an Anosov flow on $M$.
Then $X_t$ is topologically equivalent to a suspension Anosov flow on $M$.
\end{theorem}

\begin{remark}\label{r.plante}
Note that $M$ always is homeomorphic to  a mapping torus of an Anosov automorphism $\Phi$ on $T^2$.
Also note that $M$ has a unique torus fibration structure (see, for instance \cite{Thu} or \cite{Fri}),
$X_t$ should be topologically equivalent to  the suspension Anosov flow of either $\Phi$ or $\Phi^{-1}$.
In particular, when $\Phi =A =\left(
                                \begin{array}{cc}
                                  2 & 1 \\
                                  1 & 1 \\
                                \end{array}
                              \right)$,
                              $\Phi$ and $\Phi^{-1}=\left(
                                                             \begin{array}{cc}
                                                               1 & -1 \\
                                                               -1 & 2 \\
                                                             \end{array}
                                                           \right)$
 are conjugate by $C=\left(
                       \begin{array}{cc}
                         0 & 1 \\
                         -1 & 0 \\
                       \end{array}
                     \right)$, therefore, in this case, $X_t$ should be topologically equivalent to  the suspension Anosov flow of $\Phi=A$.
\end{remark}

In the proof of Theorem \ref{t.main}, we will also use the following result due to Yang and the author of this paper (Theorem 1.2 of \cite{YY}),
which is devoted to classify expanding attractors on the figure-eight knot complement $N$.

\begin{theorem}\label{t.YY}
Let $Y_t$ be a smooth flow on $N$ so that  $N$ carries an expanding attractor $\Lambda$ which is the maximal invariant set of $Y_t$.
Then $Y_t$ is topologically equivalent to the DA flow $Y_t^0$ on $N$. In particular,
every orbit-preserving homeomorphism $h: (N, Y_t)\to (N, Y_t^0)$ should satisfy that $h(\Lambda)=\Lambda_0$, where $\Lambda_0$ is the expanding attractor of $(N, Y_t^0)$.
\end{theorem}

Let us briefly introduce expanding attractor $\Lambda_0$ and the DA flow $Y_t^0$. One can find more details in \cite{YY}. Let  $Y_t$ be a smooth flow on $N$ so that  $Y_t$ is transverse into $N$ along the boundary torus $T$. If the maximal invariant set of $Y_t$, $\Lambda$ is a uniformly hyperbolic attractor with topological dimension $2$, we say that $\Lambda$ is
 an \emph{expanding attractor}, and $N$ \emph{carries} $\Lambda$. Note that an expanding attractor is always an essential lamination on the underling manifold.
 We say that a periodic orbit $\gamma$ of $\Lambda$ is a \emph{boundary periodic orbit}, if there exists a separatrix  of $W^s (\gamma)\setminus \gamma$ is free, i.e. there exists a connected component $W_{loc,+}^s (\gamma)$
  of $W_{loc}^s (\gamma)\setminus \gamma$  so that $W_{loc,+}^s (\gamma)$ is disjoint with $\Lambda$. Here $W_{loc}^s (\gamma)$
  is a small tubular neighborhood of $\gamma$ in $W^s (\gamma)$. We remark that in fact the number of the boundary periodic orbits are equivalent to the number of the boundary leaves of the attractor lamination $\Lambda$.

 The  figure-eight knot complement $N$ carries a canonical expanding attractor $\Lambda_0$ through
 DA-surgery. Let us be more precise. Starting with a suspension Anosov flow on the sol-manifold
 $W_A= T^2 \times [0,1]\setminus (x,1)\cong (A(x),0)$ induced by the vector field $(0, \frac{\partial}{\partial t})$, we perform a DA-surgery on a small tubular neighborhood of the
 periodic orbit $\gamma$ associated to the origin of $T^2$. By cutting a small and suitable open tubular neighborhood of $\gamma$, one get a compact $3$-manifold $N_A$ which is homeomorphic to $N$ so that the surgeried flow $Y_t^0$  is transverse to $\partial N_A$ and the maximal invariant set of $Y_t^0$ is an expanding attractor $\Lambda_0$.  For simplicity, we set $N_A =N$ and we call $Y_t^0$ the \emph{DA flow}. The weak stable manifolds of $\Lambda_0$ form an foliation
 $\cF_{\Lambda_0}^s$ on $N$ which is transverse to $T=\partial N$. Moreover, $f_{\Lambda_0}^s= \cF_{\Lambda_0}^s \cap T$ is a $1$-foliation which is the union of two Reeb annuli so that the
 compact leaf of $f_{\Lambda_0}^s$ is slope-$\infty$.

 Theorem \ref{t.YY} and Theorem \ref{t.plante} imply the following very useful lemma during the proof of Theorem \ref{t.main}.

 \begin{lemma}\label{l.knotint}
 Let $X_t$ be an Anosov flow on some $M(r)$. If there exists a periodic orbit $\omega$ in $X_t$
 which is isotopic to a core $c(V)$ of $V$ in $M(r)$, then $r\in \ZZ$ and $X_t$ is topologically equivalent to $X_t^r$.
 \end{lemma}
 \begin{proof}
 Note that $\omega$ is a periodic orbit of $X_t$ which is isotopic to $c(V)$ in $M(r)$, then  $M(r)- U(\omega)$ is isotopic to $N$ in $M(r)$. Here $U(\omega)$ is any open tubular neighborhood of $\omega$.
 Therefore, by doing DA surgery on $(M(r), X_t)$ along $\omega$, we can assume that we get a DA flow $Y_t$
 on $N$ with maximal invariant set an expanding attractor $\Lambda$.

 Due to Theorem \ref{t.YY}, up to  topological equivalence, we can assume that   $Y_t =Y_t^0$  and
$\Lambda=\Lambda_0$. This fact is the point of the proof.
Firstly, it implies that the complement of the attractor lamination $\Lambda_0$ in $M(r)$ admits an $I$-bundle. With the position of $\Lambda_0$ in mind, one can immediately get that $r\in \ZZ$.
Secondly, when  $r\in \ZZ$, the fact further implies that  the weak unstable manifolds $W^u (\omega)$ of $X_t$
and $W^u (\gamma)$ of $X_t^r$ have the same local framing information, i.e. there are a small  annulus tubular
 neighborhood $W_{loc}^s (\omega)$ of $\omega$ in  $W^s (\omega)$ and  a small   annulus tubular neighborhood $W_{loc}^u (\gamma)$  of $\gamma$ in  $W^u (\gamma)$ so that there exists a self-homeomorphism $h$ on $M(r)$
 which satisfies that $h(W_{loc}^u (\omega))=W_{loc}^u (\gamma)$.

 Recall that the Anosov flow $X_t^r$ is obtained  by doing $r$-Dehn-Fried-Goodman surgery
along $\gamma$ on $(M(0),X_t^0)$, further with the fact that the weak unstable manifolds $W^u (\omega)$ of $X_t$
and $W^u (\gamma)$ of $X_t^r$ have the same local framing information, we have that, after doing $-r$-Dehn-Fried-Goodman surgery along $\omega$ on $(M(r),X_t)$, we can obtain a new Anosov flow $Z_t$ on $M(0)$. By Theorem \ref{t.plante}
and Remark \ref{r.plante}, $Z_t$ is topologically equivalent to $X_t^0$. Further note the fact that $\gamma$ is the unique periodic orbit of $X_t^0$ so that the path closure of $\gamma$ in $M(r)$ is homeomorphic to $N$.
Then there exists an orbit-preserving homeomorphism $h$ between $Z_t$ and $X_t^0$ so that $h(\omega)=\gamma$. Hence, $X_t$ is also topologically equivalent to the Anosov flow obtained  by doing $r$-Dehn-Fried-Goodman surgery
along $\gamma$ on $(M(0),X_t^0)$. Therefore, $X_t$ is topologically equivalent to $X_t^r$.
  \end{proof}

  Let $X_t$ be an Anosov flow on a closed orientable $3$-manifold $M$, $\cF^s$ and $\cF^u$
  be the stable and the unstable foliations of $X_t$ respectively. Since $M$ is orientable, for a small Poincare section of any
  periodic orbit of $X_t$, the corresponding first return map  should be orientable. Therefore,
  $\cF^s$ is transversely orientable if and only if $\cF^u$ is transversely orientable. Then, we can define that the Anosov flow $X_t$ is \emph{coorientable} (\emph{non-coorientable}, resp.) if $\cF^s$ is (is not, resp.) transversely orientable. The result in the following lemma maybe standard, but we do not find any references about it. Therefore, we prove it here.

  \begin{lemma}\label{l.infper}
  Let $X_t$ be a non-coorientable transitive Anosov flow on a closed orientable $3$-manifold $M$.
  Then there are infinitely many periodic orbits of $X_t$ so that each of their weak stable/unstable manifolds is not orientable.
  \end{lemma}
  \begin{proof}
  Since $X_t$ is a non-coorientable Anosov flow, there exists a periodic orbit $\gamma$ of $X_t$
  so that $W^s (\gamma)$ is not orientable. Pick a point $P\in \gamma$ and a small disk section $\Sigma$
  of $X_t$ so that $P$ is in the interior of  $\Sigma$. Let $\Sigma_0$ be a smaller disk neighborhood of $P$ in $\Sigma$ so that,
  \begin{enumerate}
    \item $\Sigma_0$ is a  rectangle in the sense that $\cF^s \cap \Sigma_0$ and $\cF^u \cap \Sigma_0$ induce a $1$-dimensional product bi-foliation $(f^s, f^u)$ on $\Sigma_0$;
    \item  the first return map $R_1$ along the flowlines of $X_t$ satisfies that $R_1 (P)=P$ and $R_1 (\Sigma_0) \subset \Sigma$.
  \end{enumerate}

  Then, naturally  $R_1 (W_{\Sigma_0}^s (P)) \subset W_{\Sigma_0}^s (P)$ which is orientation reserving, where $W_{\Sigma_0}^s (P)$ is the interval leaf in $f^s$ containing $P$. Note that
  $W_{\Sigma_0}^s (P)$ exactly is the arc in $W^s (\gamma) \cap \Sigma_0$ which contains $P$.  We can  similarly define $W_{\Sigma_0}^u (P)$ so that $(R_1)^{-1} (W_{\Sigma_0}^u (P)) \subset W_{\Sigma_0}^u (P)$.

  Since $X_t$ is a transitive Anosov flow, there exists
a point $x\in W_{\Sigma_0}^u (P)$ so that the orbit of $x$, $\mbox{Orb} (x)\subseteq W^u (\gamma)\cap W^s (\gamma)$. Denote by $\tau$ the positive number so that $X_{\tau}(x) \in W_{\Sigma_0}^s (P)$ and $X_{(0,\tau)} (x) \cap W_{\Sigma_0}^s (P) =\emptyset$.
  Note that there are only finitely many points in the set $\mbox{Orb} (x) \cap (\Sigma_0 - W_{\Sigma_0}^s (P) \cup W_{\Sigma_0}^u (P))$. Therefore, we can choose $\Sigma_0$ small enough so that $X_{(0,\tau)} (x) \cap \Sigma_0 =\emptyset$. Below we assume that $X_{(0,\tau)} (x) \cap \Sigma_0 =\emptyset$.

By Shadowing lemma (see for instance \cite{PW}), there exists a point $Q_0\in \Sigma_0$ and a new periodic orbit
$\omega_0$ shadowed by the pseudo-orbit $X_{[0,\tau]} (x)$ so that $Q_0 \in \omega_0$.
Then we can choose a  rectangle neighborood $\Sigma_1$ of $Q_0$ in $\Sigma_0$
so that there is a product bi-foliation $(f_1^s, f_1^u)$ on $\Sigma_1$ induced by $(\cF^s\cap \Sigma_1, \cF^u \cap \Sigma_1)$ which satisfies that,
\begin{enumerate}
  \item  each leaf in $f_1^s$ is a leaf in $f^s$;
  \item  each leaf $l_1^u$ in $f_1^u$ satisfies that $R_2 (l_1^u)$ is a leaf in $f^u$,
  where $R_2$ is the first return map of $\Sigma_1$ to $\Sigma_0$.
\end{enumerate}
Notice that by Shadowing lemma, the periodic orbit $\omega_0$ and the pseudo-orbit  $X_{[0,\tau]} (x)$  are close enough.  $W^s (\omega_0)$ is orientable if and only if
$R_2$ is orientation preserving on $f_1^s$.

For every  $n\in \NN$, we define a pseudo-orbit $X_{[0,n\tau_0+\tau]} (x)$. Here $\tau_0 >0$ is the periodic time of $\gamma$, and we also assume that for every point $y\in \Sigma_0$, $X_{\tau_0} (y) = R_1 (y)$. This assumption works up to topological equivalence.
We remark that $X_{\tau} (x)= R_2 (x) \in W_{\Sigma_0}^s (P)$. By Shadowing lemma once more, there exists a point $Q_n\in \Sigma_0$ and a new periodic orbit
$\omega_n$ shadowed by the pseudo-orbit $X_{[0,n\tau_0+\tau]} (x)$ so that $Q_n \in \omega_n$.
Moreover, by the shadowing property and the facts $X_{\tau} (x)= R_2 (x)$ and $X_{\tau_0}=R_1$ on $\Sigma_0$, $W^s (\omega_n)$ is orientable if and only if $R_1^k \circ R_2$  is orientation preserving on $f_1^s$.
Recall that $R_1$ is orientation reserving on $f^s$, therefore, the orientations of $W^s (\omega_n)$ and
$W^s (\omega_{n+1})$ are different for every $n\in \NN$.
This implies that in the periodic orbit set $\{\omega_n| n\in \NN\}$, there are infinitely many
periodic orbits so that each of their weak stable manifolds is not orientable.
 \end{proof}

%% file: proof.tex
\section{The proof of Theorem \ref{t.main}}\label{s.proof}
In this section, we will prove  Theorem \ref{t.main} by discussing each branched surface in
Schwider's list.  For convenience, we will divide  the proof into $4$ classes, which
are associated to Section \ref{ss.Bd}, \ref{ss.BBII}, \ref{ss.BSI} and \ref{ss.BSII} respectively.

\subsection{Branched surfaces with disk compact leaves}\label{ss.Bd}
There are five types of branched surfaces in Schwider's list, so that each of their sectors in $V$
is a disk. They are $B_1$, $B_2$, $B_3$, $B_4$ and $B_9^M$. The purpose of this subsection is to show that none of them carries any Anosov lamination. More precisely, we have

\begin{proposition}\label{p.disktype}
There does not exist any Anosov flow $X_t$ on some $M(r)$, so that the stable foliation $\cF^s$ of $X_t$
is fully carried by one of $B_1$, $B_2$, $B_3$, $B_4$ and $B_9^M$.
\end{proposition}

The idea to prove Proposition \ref{p.disktype} is quite simple: by understanding the topology and
the vertical/horozontal boundary of the  manifold $W(B) = M(r) - \mbox{int} (N(B))$, where $B$ is anyone
of the five types of branched surfaces above, we can  always find an obstruction to fully carry Anosov foliation. Let us by building some lemmas to describe some properties of $W(B)$.

\begin{lemma}\label{l.B1}
$W(B_1)$ is homeomorphic to $T^2 \times [0,1]$, $\partial_h W(B_1) =T^2 \times \{0,1\}$
and $\partial_v W(B_1)=\emptyset$.
\end{lemma}
\begin{proof}
See Section \ref{sss.Q1}, $B_1$ is homeomorphic to a fiber torus on the  sol manifold $M(0)$. Therefore, $W(B_1)$ is homeomorphic to $T^2 \times [0,1]$. Since there does not exist any branch
locus, $\partial_h W(B_1) =T^2 \times \{0,1\}$ and $\partial_v W(B_1)=\emptyset$.
\end{proof}

\begin{lemma}\label{l.B2B4}
Each of $W(B_2)$ and $W(B_4)$ is homeomorphic to a genus two handlebody.
\end{lemma}
\begin{proof}
Recall that $B_2$ is the union of $B_2^n$ and $B_2^v$ where $B_2^n= B_2 \cap N$ and $B_2^v =B_2 \cap V$.
$W(B_2)= (N- \mbox{int} (N(B_2)))\cup (V- \mbox{int} (N(B_2)))= (N- \mbox{int} (N(B_2^n)))\cup (V- \mbox{int} (N(B_2^v)))$.
See the left part of Figure \ref{f.5B2}, $N- \mbox{int} (N(B_2^n))$ is a compact $3$-manifold with one boundary connected component
so that the circle $c\cup d$ is one of its deformation retracts. Therefore, $N- \mbox{int} (N(B_2^n)))$ is homeomorphic to a solid torus $V_0$.
\begin{figure}[htp]
\begin{center}
  \includegraphics[totalheight=3.2cm]{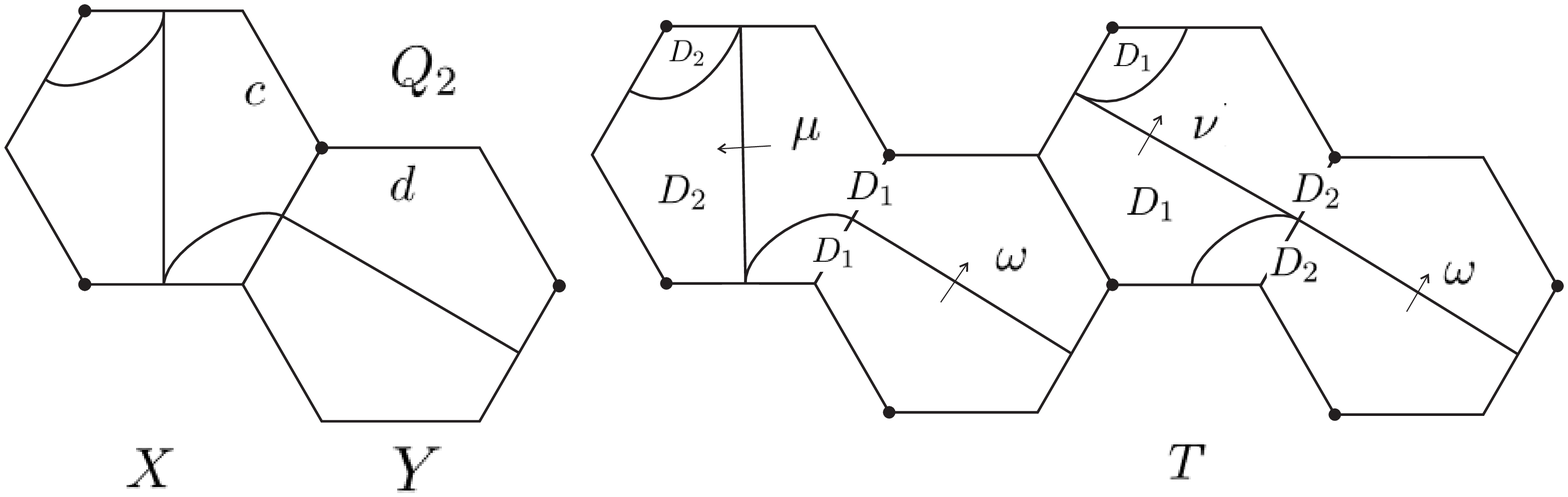}\\
 \caption{figures for the proof of Lemma \ref{l.B2B4} in the case of $B_2$}\label{f.5B2}
\end{center}
\end{figure}

Recall that for each essential lamination $\lambda$ fully carried by $B_2$, $B_2 \cap \partial N$ only
carry $(\pm w, \pm u \mp v)$-circles. Therefore, $B_2^v$ is an immersed meridian disk $D'$ so that the boundary circle is carried by the train track $B_2 \cap \partial N$. Hence, $V-B_2^v$ is an open $3$-ball.
These facts induce that the interior of $V- \mbox{int} (N(B_2^v))$ is also an open $3$-ball $B^3$ so that
$V- \mbox{int} (N(B_2^v))$  meets $N- \mbox{int} (N(B_2^n))$ at the union of two open disks $D_1$ and $D_2$ on $T$. See the right part of Figure \ref{f.5B2} for $D_1$ and $D_2$ on $T$.

Therefore, $W(B_2)= M(r)- \mbox{int}(N(B_2)) = (N- \mbox{int} (N(B_2^n)))\cup (V- \mbox{int} (N(B_2^v)))$ is equivalent to
the manifold obtained by gluing the solid torus $V_0$ and the $3$-ball $B^3$ along two disks $D_1$ and $D_2$, which is homeomorphic to a genus two handlebody.
\begin{figure}[htp]
\begin{center}
  \includegraphics[totalheight=3.2cm]{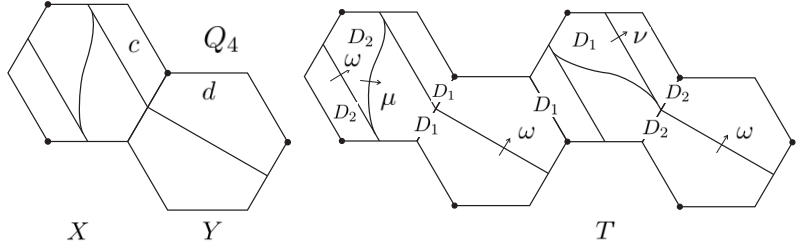}\\
 \caption{figures for the proof of Lemma \ref{l.B2B4} in the case of $B_4$}\label{f.5B4}
\end{center}
\end{figure}

See Figure \ref{f.5B4} for the case of $B_4$, one can similarly prove the lemma in the case of $B_4$,
we leave the details to an interested reader.
\end{proof}


\begin{lemma}\label{l.B3}
$B_3$ is homeomorphic to a Klein bottle so that $\partial_h W(B_3)\cong T^2$ and $\partial_v W(B_3)= \emptyset$.
\end{lemma}
\begin{proof}
See Section \ref{sss.Q1}, $B_3$ is homeomorphic to a Klein bottle. Further note that $M(r)$ is orientable,
so, as a small tubular neighborhood of $B_3$, $N(B_3)$ can be endowed with a twisted I-bundle
over $B_3$. Therefore, $\partial_h W(B_3)= \partial_h N(B_3)$ is homeomorphic to $T^2$ and $\partial_v W(B_3)= \emptyset$.
\end{proof}

\begin{lemma}\label{l.B9M}
$W(B_9^M)$ is homeomorphic to a genus four handlebody.
\end{lemma}
\begin{proof}
The proof is similar to the proof of Lemma \ref{l.B2B4}, so, we will omit some details.
At first, see the left part of Figure \ref{f.5B9M}, $N- \mbox{int} (N(B_9^M)))$ is a compact $3$-manifold with one boundary connected component
so that the edge $c$ is one of its deformation retracts. Therefore, $N- \mbox{int} (N(B_9^M))$ is homeomorphic to a $3$-ball. Furthermore, by the same argument to the proof of
Lemma \ref{l.B2B4}, the interior of $V-(B_9^M)$ is an open $3$-ball.
Secondly, see the right part of Figure \ref{f.5B9M}, one can observe that $T-B_9^M$ is the union of five open disks $D_1, \dots, D_5$.
\begin{figure}[htp]
\begin{center}
  \includegraphics[totalheight=3.2cm]{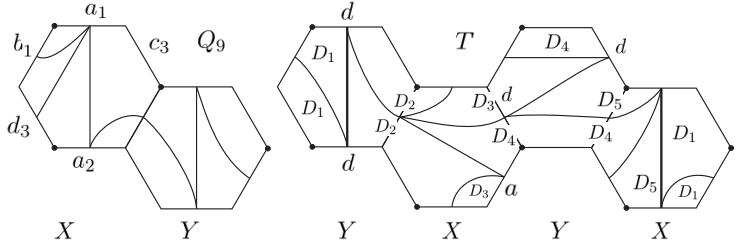}\\
 \caption{figures for the proof of Lemma \ref{l.B9M}}\label{f.5B9M}
\end{center}
\end{figure}

Therefore, $M-B_9^M$ is homeomorphic to the manifold obtained by gluing two $3$-balls
along $5$ boundary disks, which is homeomorphic to a genus four handlebody.
\end{proof}

Now we can finish the proof of Proposition \ref{p.disktype}.

\begin{proof}[The proof of Proposition \ref{p.disktype}]
By Lemma \ref{l.B1}, \ref{l.B2B4},  \ref{l.B3} and \ref{l.B9M},
for each branched surface $B$ in $B_1$, $B_2$, $B_3$, $B_4$ and $B_9^M$,
$W(B)$ is connected and is not in the three types of manifolds listed in
Proposition \ref{p.Anobracomp}. Therefore, $B$ does not carry any Anosov lamination. The conclusion of the proposition is followed.
\end{proof}

\subsection{Basic branched surfaces with type II annular sector} \label{ss.BBII}
In Section \ref{s.Schwiderlist}, we have introduced five types of branched surfaces in Schwider's list, so that
each of them contains some type II annular sectors, but does not contain any type I annular sector,
and also does not exist any two  type II annular sectors which can be abstractly obtained by splitting a type II sector of another branched surface. They are $B_6$, $B_7$, $R_7$, $B_8$ and $B_9$. We will discuss the Anosov laminations carried by these branched surfaces.
Now we need to describe them in two cases which correspond to the following two propositions.

\begin{proposition}\label{p.typeII1}
Let $B$ be one of $B_6$, $B_7$, $B_8$ and $B_9$.
If some $M(r)$ admits an Anosov flow $X_t$ so that the stable foliation $\cF^s$ of $X_t$
is fully carried by $B$, then $r\in \ZZ$, and there exists a periodic orbit $\omega$ of $X_t$ so that $\omega$ is isotopic to a core $c(V)$ of $V$ in $M(r)$.
\end{proposition}

\begin{proposition}\label{p.typeII2}
There does not exist any $M(r)$ which carries an Anosov flow $X_t$ so that
the Anosov foliation is fully carried by $R_7$.
\end{proposition}

\begin{proof}[The proof of Proposition \ref{p.typeII1}]
By Section \ref{ss.Q6}, \ref{ss.Q7}, \ref{ss.Q8} and \ref{ss.Q9}, one can easily
observe that $B$ contains an  annular sector $\Sigma$ so that each of its two end-loops on $T$ is slope-$\infty$. We remark that the proof will  strongly depends on this observation.

Assume that $\cL$ is an Anosov lamination on $M(r)$ ($r=\frac{q}{p}$) so that,
 \begin{enumerate}
   \item $\cL$ is obtained by splitting finitely many leaves of $\cF^s$;
   \item $\cL$ if fully carried by $B$.
 \end{enumerate}
Since $\cL$ is fully carried by $B$, then there exists a  leaf $l$ of $\cL$
which contains an annular plaque carried by the glued annular sector $\Sigma$ of $B$.
Let $l^s$ be the leaf in $\cF^s$ so that $l$ is a splitting connected component of $l^s$, and $\omega$ be the periodic orbit of $X_t$ so that $\omega \subset l^s$.
 Note that
the boundary of  $N\cup \Sigma$ is the union of two tori.
Here, we think that $\Sigma$ provides two annuli in the boundary of $N\cup \Sigma$ by its two sides.
In the manifold $M(r)$, at most one  boundary torus  bounds an \emph{exceptional solid torus}
$V_0$, i.e. a core of $V_0$, $c(V_0)$, and a core of $\Sigma$, $c(\Sigma)$, are not isotopic in $V$.\footnote{We remark that exceptional solid torus will be used several times below. Moreover, under the same assumptions and notations, if $c(V_0)$ and $c(\Sigma)$ are isotopic in $V$, we say that $V_0$ is a \emph{regular solid torus}.}
 Recall that $r=\frac{q}{p}$, it is easy to compute that $c(\Sigma)$ is homotopic to
either $c(V_0)^p$ or $c(V_0)^{-p}$ in $V$.
Note that a small tubular neighborhood of $\omega$, $W_{loc}^s (\omega)$, can be tamely embedded into $l$, and $l$ contains an annulus plaque carried by  $\Sigma$, then
$\omega$ is homotopic to either $c(\Sigma)$ or $c(\Sigma)^{-1}$.
Therefore, $\omega$ is homotopic to either $c(V_0)^p$ or $c(V_0)^{-p}$.
By  a classical result due to Fenley \cite{Fen2}, Theorem \ref{t.fenley},
 $\omega$ is either primitive in $\pi_1 (M(r))$
or $\omega$ is homotopic to $\alpha^2$ where
$\alpha$ is a primitive element in $\pi_1 (M(r))$, and moreover, in the later case, $X_t$ is non-coorientable.
Then either $p=1$, or $p=2$ and $X_t$ is non-coorientable.
Further by Lemma \ref{l.infper} which implies that $\cF^s$ contains infinitely many non-orientable leaves and the fact that $\cL$ can be obtained by splitting finitely many leaves of $\cF^s$,
$\cL$ must be not transversely orientable. Recall that $\cL$ is fully carried by $B$. Then by Corollary \ref{c.B6789nocarry}, $\cL$ is transversely orientable. We get a contradiction, which means that $r\in \ZZ$.

When $r\in \ZZ$, each of the two  boundary tori of  $N\cup \Sigma$ bounds a regular solid torus in $M(r)$.
Further notice that the annular sector $\Sigma$ satisfies that each of its two end-loops on $T$ is slope-$\infty$. These two facts imply that
a core $c(\Sigma)$ of the annulus plaque $\Sigma$ is isotopic to a core $c(V)$ of $V$ in $M(r)$.
Recall that up to isotopy, $\omega$ can be embedded into $\Sigma$, further notice the fact that
up to isotopy, there is a unique essential simple closed curve on an open annulus,  therefore $c(\Sigma)$ and $\omega$ are isotopic in $M(r)$.
In summary, the periodic orbit $\omega$ of $X_t$ is also isotopic to a core $c(V)$ of $V$ in $M(r)$.
\end{proof}

\begin{proof}[The proof of Proposition \ref{p.typeII2}]
Recall that $R_7$ is obtained by gluing three annular sectors $\Sigma_1$, $\Sigma_2$ and $\Sigma_3$ to $B_7^n$ (see Figure \ref{f.B7NcapT}) along $f$ and $g$, $g$ and $h$, and $f$ and $h$ respectively.
One can easily observe that there exists a connected component $V_0$ in the  manifold
$W(R_7)=M(r)-\mbox{int}(N(R_7))$ which is a solid torus  so that $\partial_v V_0$ is the union of three circles with slope-$\infty$. Therefore, $W(R_7)$ does not carry any I-bundle so that
the I-bundle is coherent to $\partial_v W(R_7)$. By Proposition \ref{p.Anobracomp}, $R_7$ does not fully carry any
Anosov foliation.
\end{proof}

\subsection{Branched surfaces with type I annular sectors}\label{ss.BSI}
In Section \ref{s.Schwiderlist}, we have introduced $19$ types of branched surfaces in Schwider's list, so that
each of them contains a vacant type I annular sector $\Sigma$ in $V$ on $M_r$. They are $B_5$,
$B_6^I$ ($2$), $B_7^I$ ($3$), $R_7^I$ ($3$), $B_7^{\ast}$, $B_7^{\ast \ast}$ ($6$),
$B_8^{III}$, $B_{10}$ and $B_{11}$.

First of all, we build the following lemma.

\begin{lemma}\label{l.typeI}
Let $B$ be one of the $19$ types of branched surfaces in Schwider's list which contain a type I annular sector $\Sigma$, then in the  manifold $W(B)=M(r) -\mbox{int}(N(B))$, there is a solid torus connected component $V_0$ which is adjacent to one annular side of $\partial_h (N(\Sigma))$, namely $A_0$, so that there exists a meridian circle $m$ of $V_0$
in $\partial V_0$ which intersects to a core of $A_0$ twice.
\end{lemma}
\begin{proof}
Since one of the two sides of $\Sigma$  is vacant, then there is a solid torus connected component $V_0$ of $W(B)$
so that it is adjacent to one side of $\partial_h (N(\Sigma))$,
which is an annulus. This annulus is named by $A_0$ suggested in the lemma.

Furthermore, notice that $B$ fully carries an Anosov lamination, and $\partial_v V_0$ is connected,
item $2$ of Proposition \ref{p.Anobracomp} induces that $V_0$ must satisfies that there exists a meridian circle $m$ of $V_0$
in $\partial V_0$ which intersects to a core of $A_0$ exactly twice.
\end{proof}

\begin{proposition}\label{p.typeI}
Let $B$ be one of the $19$ types of branched surfaces in Schwider's list which contain a type I annular sector $\Sigma$ in $V$ on the corresponding $3$-manifold $M(r)$. Then $B$ does not fully carry any Anosov lamination.
\end{proposition}
\begin{proof}
By Lemma \ref{l.typeI}, $V_0$ is an exceptional solid torus in $M(r)$. Therefore there does not
exist any other exceptional solid torus in $W(B)$. This implies that $V_0$ is a core solid torus in
$V$, i.e. $\partial V$ is a deformation retract of the path closure of $V-V_0$.
Therefore, $N_0$ and $N$ are isotopic in $M(r)$, where $N_0= M(r)- \mbox{int}(V_0)$.

Let $\lambda$ be an Anosov lamination fully carried by $B$ in $M(r)$ and $l$ be the boundary leaf
which contains an annulus plaque carried by $\Sigma$. Without loss of generality, we can assume that
there exists  a periodic orbit $\gamma$ of the associated Anosov flow $X_t$ so that
$l$ is obtained by splitting $W^s (\gamma)$ on the stable foliation of $X_t$. By Proposition \ref{p.Anobracomp}, $\gamma$ is isotopic to
the core of $V_0$.

 By doing DA surgery on $X_t$ along $\gamma$, we can get a flow $Y_t$
on $N_0$ so that,
\begin{enumerate}
  \item $N_0$ is homeomorphic  to $N$;
  \item $Y_t$ is transverse to $\partial N_0$;
  \item the maximal invariant set of $Y_t$ on $N_0$, $\Omega (Y_t)$, is an expanding attractor
  with a unique boundary periodic orbit.
\end{enumerate}
Item $1$ above is a consequence of the fact that $N_0$ and $N$ are isotopic in $M(r)$, which is proved in the first paragraph of the proof.
Item $3$ can be followed by the fact that $l$
is obtained by splitting $W^s (\gamma)$ on the stable foliation of $X_t$.

The existence of $Y_t$ on the figure-eight knot complement $N_0$ conflicts to
the conclusion of  Theorem \ref{t.YY} (see also Theorem 1.2 of \cite{YY}), which says that, up to topological equivalence, $N_0$ only carries a unique
expanding attractor, and moreover, there exist two boundary periodic orbits on this attractor.
Therefore, $B$ does not fully carry any Anosov lamination.
\end{proof}

\subsection{Branched surfaces obtained by splitting type II annular sectors}\label{ss.BSII}
We are left to discuss the branched surfaces in Schwider's list which abstractly
are obtained by splitting a type II annular sector on one of the branched surfaces listed in Section \ref{ss.BBII}. They are: $B_6^{II}$, $B_7^{II}$ ($2$), $R_7^{II}$ ($3$), $B_8^{II}$ ($2$) and $B_9^{II}$. In fact, none of them fully carries any Anosov lamination.
More precisely, we have

\begin{proposition}\label{p.spl}
Let $B$ be one of $B_6^{II}$, $B_7^{II}$ ($2$), $R_7^{II}$ ($3$), $B_8^{II}$ ($2$) and $B_9^{II}$ on the corresponding  $3$-manifold $M(r)$. Then $B$ does not fully carry any Anosov lamination.
\end{proposition}

\begin{proof}
Since $B$ abstractly can be obtained by splitting a type II annular sector on one of the branched surfaces $B'$ listed in Section \ref{ss.BBII}. Let $\Sigma_1$ and $\Sigma_2$ be the two annular sectors in $B$ obtained by splitting an annular sector $\Sigma'$ of $B'$.
 Then
 in the  manifold $W(B)=M(r) -\mbox{int}(N(B))$, there is a solid torus connected component $V_0$ which is adjacent to an annular component of $\partial_h (N(\Sigma_1))$, namely $A_1$, and an annular component of $\partial_h (N(\Sigma_2))$, namely $A_2$. It is important to observe that $V_0$ is the unique exceptional solid torus in
 $M(r)- (N\cup \Sigma_1\cup \Sigma_2)$. since otherwise in the view point of embedding, $B$  also can be obtained by splitting $B'$. Then in this case, we can think $B=B'$.\footnote{One can find a similar statement in Page $89$ of \cite{Sch}.}
 Therefore, there exists a meridian circle $c_m$  in the boundary of the closure of $V_0$ which intersects to a core circle of $A_1$ more than once, then $W(B)$ does not carry any I-bundle so that the I-bundle is coherent to $\partial_v W(B)$. By proposition \ref{p.Anobracomp}, $B$ does not carry any Anosov lamination.
\end{proof}

\subsection{End of the proof}\label{ss.proof}
Now we can finish the proof of the main theorem of the paper.

\begin{proof}[The proof of Theorem \ref{t.main}]
Let $X_t$ be an Anosov flow on $M(r)$ and $\cF^s$ be the corresponding  stable foliation.
Then $\cF^s$ should be fully carried by one of branched surfaces in Schwider's list, say $B$.
Due to Proposition \ref{p.disktype}, \ref{p.typeII1}, \ref{p.typeII2}, \ref{p.typeI} and \ref{p.spl}, we have
\begin{enumerate}
  \item $r$ should satisfy that $r\in \ZZ$;
  \item when $r\in \ZZ$, $B$ can only  be one of $B_6$, $B_7$, $B_8$, $B_9$. Moreover, there exists a periodic orbit $\omega$ in $X_t$ which is isotopic to a core $c(V)$ of $V$ in $M(r)$.
\end{enumerate}

As the direct consequence of item 1 above, when $r\notin \ZZ$, $M(r)$ does not carry any Anosov flow. Item 2 of the theorem is proved.

When $r\in \ZZ$, by item 2 above and Lemma \ref{l.knotint}, $X_t$ is topologically equivalent to $X_t^r$. Item 1 of the theorem is proved.
\end{proof}